%% file: main.tex
\documentclass{article}
\usepackage[utf8]{inputenc}
\usepackage{geometry}
\usepackage{amsmath,amsthm,amssymb,amsfonts}
\usepackage{comment}
\usepackage{wasysym}
\usepackage{color}
\usepackage[hidelinks]{hyperref}
\usepackage{enumerate}
\usepackage{accents}
\usepackage{graphicx}
\usepackage{tikz}
\usepackage{standalone}
\usepackage{subcaption}
\usepackage{authblk}
\usepackage{cleveref}

\newcommand{\R}{\mathbb{R}}

\newcommand{\Z}{\mathbb{Z}}
\newcommand{\N}{\mathbb{N}}
\newcommand{\cF}{\mathcal{F}}
\renewcommand{\P}{\mathbb{P}}
\newcommand{\E}{\mathbb{E}}
\newcommand{\1}{\mathbf{1}}

\newcommand{\limn}{\lim_{n\rightarrow\infty}}
\newcommand{\limm}{\lim_{m\rightarrow\infty}}
\renewcommand{\l}{\ell}
\newcommand{\heta}{\widehat\eta}
\newcommand{\teta}{\widetilde\eta}

\newcommand{\esing}{\eta^{\infty}}

\newcommand{\en}{\eta^n}

\newcommand{\esc}{\text{ escapes }}
\newcommand{\nesc}{\text{ does not escape }}
\newcommand{\issaw}{\text{ is a self-avoiding path }}

\providecommand{\MR}[1]{}

\makeatletter
\DeclareRobustCommand{\cev}[1]{%
  \mathpalette\do@cev{#1}%
}
\newcommand{\do@cev}[2]{%
  \fix@cev{#1}{+}%
  \reflectbox{$\m@th#1\vec{\reflectbox{$\fix@cev{#1}{-}\m@th#1#2\fix@cev{#1}{+}$}}$}%
  \fix@cev{#1}{-}%
}
\newcommand{\fix@cev}[2]{%
  \ifx#1\displaystyle
    \mkern#23mu
  \else
    \ifx#1\textstyle
      \mkern#23mu
    \else
      \ifx#1\scriptstyle
        \mkern#22mu
      \else
        \mkern#22mu
      \fi
    \fi
  \fi
}
\makeatother

\DeclareMathOperator{\TV}{TV}
\DeclareMathOperator{\Cov}{Cov}
\DeclareMathOperator{\Var}{Var}
\DeclareMathOperator{\SAW}{SAW}
\DeclareMathOperator{\SLE}{SLE}

\newtheorem{theorem}{Theorem}
\newtheorem{corollary}[theorem]{Corollary}
\newtheorem{lemma}[theorem]{Lemma}

\newtheorem{proposition}[theorem]{Proposition}
\newtheorem{conjecture}[theorem]{Conjecture}

\theoremstyle{definition}
\newtheorem{definition}[theorem]{Definition}
\newtheorem{remark}[theorem]{Remark}

\numberwithin{equation}{section}
\numberwithin{theorem}{section}

\begin{document}

\title{Two-sided infinite self-avoiding walk in high dimensions}
\date{\vspace{-5ex}}
\author{Maarten Markering}
\affil{University of Cambridge}
\maketitle
\let\thefootnote\relax\footnotetext{Date: October 6, 2026}
\let\thefootnote\relax\footnotetext{Email: maartenmarkering@outlook.com}
\let\thefootnote\relax\footnotetext{2020 Mathematics Subject Classification: 60K35, 60F15, 82B27, 82B41, 37A25}
\let\thefootnote\relax\footnotetext{Key words: self-avoiding walk, law of large numbers, ergodicity, lace expansion}

\begin{abstract}
In this paper, we consider self-avoiding walk (SAW) on $\Z^d$ for $d\geq5$. We construct the \emph{two-sided infinite SAW} and use it to prove pattern theorems for the self-avoiding walk. We show that infinite two-sided SAW is the infinite-shift limit of infinite one-sided SAW and the infinite-size limit of finite two-sided SAW. We then prove that for every pattern $\zeta$, the fraction of times that $\zeta$ occurs in the SAW converges to the probability that the two-sided infinite SAW starts with $\zeta$. The convergence is in probability for the finite SAW and almost surely for the infinite SAW. For the proof, we show that infinite SAW is ergodic using a coupling technique. At the end of the paper, we pose a conjecture regarding the existence of infinite SAW in low dimensions. We show that this conjecture is true in high dimensions, thus giving a new proof for the existence of infinite SAW for $d\geq5$. The proofs in this paper rely only on the asymptotics for the number of self-avoiding paths and the SAW two-point function. Although these results were shown by Hara and Slade using the lace expansion, the proofs in this paper do not use the lace expansion.
\end{abstract}

\tableofcontents

\section{Introduction}\label{sec:intro}
Self-avoiding walk was introduced by Flory and Orr in \cite{flory53principles, orr47statistical} as a model for long polymer chains. Despite the simplicity of its definition, many important questions still remain unanswered. In this paper, we focus on the \emph{infinite self-avoiding walk}, introduced by Lawler in \cite{lawler89infinite}. We define \emph{two-sided infinite self-avoiding walk} and use it to prove several \emph{pattern theorems} in high dimensions, improving on Kesten's celebrated results from \cite{kesten63number} in the `60s.

\subsection{Model}\label{sec:model}
We first recall the definition of the self-avoiding walk. A \emph{self-avoiding path} of length $n\in\N\cup\{\infty\}$ is a nearest-neighbour path $\omega=(\omega(0),\ldots,\omega(n))$ in the graph~$\Z^d$ that does not visit any vertex twice, i.e., $i\neq j$ implies $\omega(i)\neq\omega(j)$. For $n\in\N$, write $\SAW_n$\phantomsection{}\label{def:SAWn} for the set of all self-avoiding paths $\omega$ of length $n$ such that $\omega(0)=0$, and for $x\in\Z^d$, write $\SAW_{n}(x)$\phantomsection{}\label{def:SAWn(x)} for the subset of $\SAW_n$ of paths $\omega$ such that $\omega(n)=x$, and let $\SAW_\infty$\phantomsection\label{def:SAWinfty} be the set of finite, one-sided infinite and two-sided infinite nearest-neighbour self-avoiding paths $\omega$ in $\Z^d$ such that $\omega(0)=0$. This set is defined this way to ensure that finite SAW, one-sided infinite SAW and two-sided infinite SAW all live in the same space. Lastly, we define the quantities $c_n=|\SAW_n|$\phantomsection\label{def:cn} and $c_n(x)=|\SAW_n(x)|$.\phantomsection\label{def:cn(x)} 

\begin{definition}[Self-avoiding walk]\label{def:etan}
The \emph{self-avoiding walk} (SAW) $\eta^n$ of length $n$ started from 0 is the random path that has as its law the uniform measure on $\SAW_n$, i.e., for $\omega\in\SAW_n$, 
\begin{equation}
    \P(\eta^n=\omega)=c_n^{-1}.
\end{equation}
\end{definition}

We give a very brief summary of the results on SAW that we need in this paper. We refer the interested reader to \cite{bauerschmidt12lectures, madras93selfavoiding} for more background. An important question in the study of the self-avoiding walk is counting the number of self-avoiding paths~$c_n$. It follows from Fekete's lemma and submultiplicativity of~$c_n$ that there exists a constant $\mu=\mu(d)$ such that
\begin{equation}\label{eq:defmu}
    \mu:=\limn (c_n)^{1/n}.
\end{equation}
The constant $\mu$ is called the \emph{connective constant} of the lattice $\Z^d$. For $d\geq5$, much more precise asymptotics are known for $c_n$. In lower dimensions, only very imprecise upper and lower bounds have been shown.

A central object in this paper is the \emph{infinite self-avoiding walk}. Note that the laws of $\eta^n$ are not a consistent family of probability measures, i.e., for most $\omega\in\SAW_n$, we have
\begin{equation}
    \begin{split}
        \P(\eta^n=\omega)\neq\sum_{\substack{\zeta\in\SAW_{n+1}\\\zeta\succeq\omega}}\P(\eta^{n+1}=\zeta),
    \end{split}
\end{equation}
where $\zeta\succeq\omega$\phantomsection\label{def:succeqpaths} means that $\zeta$ is a self-avoiding path that starts with $\omega$. Consider, for example a path~$\omega\in\SAW_n$ that is ``trapped'', meaning that there does not exist $\zeta\in\SAW_{n+1}$ such that $\zeta\succeq\omega$. Then the LHS of the above equation is nonzero, whereas the RHS equals 0. Therefore, there does not exist a probability measure on infinite paths such that its $n$-step marginals equal the laws of $\eta^n$. However, it was shown by Lawler in \cite{lawler89infinite} using the lace expansion techniques developed by Slade in \cite{slade89scaling} that for all $d$ sufficiently large, there exists a random variable $\esing$ supported on the set of infinite-length self-avoiding paths such that for all $k$ and all $\zeta\in\SAW_k$,
\begin{equation}\label{eq:def-esing}
    \begin{split}
        \P(\esing[0,k]=\zeta)=\limn\P(\eta^n[0,k]=\zeta).
    \end{split}
\end{equation}
The path $\esing$ is called the infinite self-avoiding walk. It is a random variable taking values in $\SAW_\infty$. If unspecified, we take $\esing(0)=0$. The relevant sigma-algebra on $\SAW_\infty$ is the sigma-algebra generated by the events depending on a finite number of indices, which we call cylinder events. We denote this sigma-algebra by $\cF$.\phantomsection\label{def:cF} Throughout the paper, all mentioned events will always be elements of this sigma-algebra. The lace expansion for self-avoiding walk was later strengthened by Hara and Slade \cite{hara92selfavoiding} to hold for all $d\geq 5$. As a corollary, it followed that $\esing$ is well-defined for all $d\geq5$. 

\subsection{Background}
In fact, much more is known for $d\geq5$. It was shown that self-avoiding walk exhibits mean-field behaviour in a strong sense. We summarize the results we need in the following theorem.
\begin{theorem}[\cite{hara92selfavoiding, hara03critical}]\label{th:countingwalks}
For all $d\geq5$, the following are true.
\begin{enumerate}
\item There exist constants $A=A(d)$ and $a=a(d)$ such that
\begin{equation}\label{def:A-saw}
    c_n\sim A\mu^n,\qquad n\rightarrow\infty
\end{equation}
and
\begin{equation}
	G(x):=\sum_{n=0}^\infty c_n(x)\mu^{-n}\sim a\|x\|^{2-d},\qquad\|x\|\rightarrow\infty.
\end{equation}
\item There exists $D=D(d)>0$ such that $n^{-1/2}\eta^n(D^{-1}n^{-1}\cdot)$, viewed as a continuous path in $\R^d$ on the time interval $[0,1]$, converges weakly in the supremum norm to standard Brownian motion.
\item For $\zeta\in\SAW_k$, let $c_n(\zeta)$ be the number of self-avoiding paths of length $n$ starting with $\zeta$. Then
\begin{equation}
    \lim_{n\rightarrow\infty}\P(\eta^n[0,k]=\zeta)=\lim_{n\rightarrow\infty}\frac{c_n(\zeta)}{c_n}=\P(\eta^\infty[0,k]=\zeta).
\end{equation}
\end{enumerate}
\end{theorem}

The lace expansion techniques to analyse SAW have also been used to obtain strong results on other high-dimensional critical lattice models, such as critical percolation and lattice trees. Nowadays, the theory of lace expansion is very mature and we will not attempt to give a complete summary here. Instead, we refer the interested reader to \cite{slade2004lace} for more background. Although the lace expansion has allowed for remarkable results on high-dimensional lattice models, the proofs are notoriously difficult. Recently, Duminil-Copin, Panis and coauthors have established mean-field behaviour of several high-dimensional lattice models without using the lace expansion, such as weakly self-avoiding walk \cite{duminil2025alternativesaw} and spread-out percolation \cite{duminil2025alternativeperc}.

We note that dimension 4 is indeed critical for self-avoiding walk. The lace expansion no longer works below $d=5$. In dimension 4, Bauerschmidt, Brydges and Slade showed mean-field behaviour for the two-point function and susceptibility of the weakly self-avoiding walk using rigorous renormalisation group methods \cite{bauerschmidt2015critical, bauerschmidt2015logarithmic}, but the analogous results for strictly self-avoiding walk remain out of reach. In dimension 3, almost nothing is known about the qualitative behaviour of SAW. In dimension 2, it was shown by Lawler, Schramm and Werner \cite{lawler04scaling} that the scaling limit of planar SAW must equal the continuous process $\SLE_{8/3}$ if it exists and is conformally invariant. 

In dimensions $2\leq d\leq 4$, the existence of infinite self-avoiding walk is still open. The best current result in lower dimensions stems from Kesten's Pattern Theorem from `63 \cite{kesten63number}, which implies that for all $\zeta$ such that $\zeta$ is not trapped, we have
\begin{equation}
    \begin{split}
        \liminf_{n\rightarrow\infty}\P(\eta^n[0,k]=\zeta)=\liminf_{n\rightarrow\infty}\frac{c_n(\zeta)}{c_n}>0.
    \end{split}
\end{equation}
In fact, Kesten's Pattern Theorem states that for most such \emph{patterns} $\zeta$, the fraction of times for which $\zeta$ occurs in $\eta^n$ is bounded from below with very high probability. A precise statement is given in \Cref{th:kestenpattern}. Kesten's Pattern Theorem has had many consequences, such as showing that $c_{n+2}/c_n\to\mu^2$, establishing the (non-)ergodicity of certain Monte Carlo algorithms for generating self-avoiding walks \cite[Section~9.4]{madras93selfavoiding}, and showing that self-avoiding polygons are knotted with high probability \cite{sumners1988knots}. The latter has consequences in the study of polymers using self-avoiding walks.

\subsection{Results}
Although the infinite self-avoiding walk was introduced three decades ago, many basic properties still remain unknown. In this paper, we explore several aspects of the infinite self-avoiding walk. We first define \emph{two-sided infinite self-avoiding walk} as a random variable on bi-infinite self-avoiding paths and prove that this can be interpreted as the limit of both the finite and infinite self-avoiding walk ``as seen from the middle''. We then show that the infinite self-avoiding walk is \emph{ergodic} with respect to the shift operator by showing that events depending on two finite parts of the infinite self-avoiding walk that are very far away are approximately independent.

Using this asymptotic independence result, we show that for every pattern, the fraction of times at which this pattern occurs converges in probability for the finite self-avoiding walk and almost surely for the infinite self-avoiding walk. This is a high-dimensional extension of Kesten's Pattern Theorem. As a consequence, we give an alternative proof of the existence of one-sided infinite self-avoiding walk. The methods in this paper only rely on \Cref{th:countingwalks} and do not use the lace expansion explicitly. Thus, if recent work on mean-field behaviour for the weakly self-avoiding walk and spread-out percolation without the lace expansion is extended to self-avoiding walk, the results from this paper carry over without any issues.

Throughout this paper, we will make constant use of the results of \Cref{th:countingwalks}. To ease notation, we will simply \textbf{assume $d\geq5$ for the rest of the paper} and will suppress this assumption in the statements of all the results. We now state the main results.

\subsubsection{Two-sided infinite self-avoiding walk}
Let $\eta^{m,n}$\phantomsection\label{def:etamn} be \emph{two-sided, finite SAW} of length $m$ and $n$, i.e., $\eta^{m,n}(-k)=\eta_1^m(k)$ and $\eta^{m,n}(k)=\eta_2^n(k)$ for $k\geq 0$, where $\eta_1^m$ and $\eta_2^n$ are finite SAWs starting from 0 conditioned not to intersect. Let $T$\phantomsection\label{def:leftshift} be the left-shift operator on $\SAW_{\infty}$, i.e., $(T\omega)(k)=\omega(k+1)-\omega(1)$. Note that $T^m\eta^{m+n}\overset{d}{=}\eta^{m,n}$. Throughout the paper, $\zeta$ denotes some finite self-avoiding path.
\begin{definition}\label{def:twosidedsaw}
Let $\esing_1,\esing_2$ be two independent infinite self-avoiding walks started from 0 conditioned such that $\esing_1[1,\infty)\cap\esing_2[1,\infty)=\varnothing$. The \emph{two-sided infinite self-avoiding walk} is defined as the bi-infinite random path $\heta$ given by $\heta(-n)=\esing_1(n)$ and $\heta(n)=\esing_2(n)$.
\end{definition}
A priori, it is not clear whether $\heta$ is well-defined. Indeed, we may have $\P(\esing_1[1,\infty)\cap\esing_2[1,\infty)=\varnothing)=0$. Our first theorem says that this is not the case. Moreover, $\heta$ is the infinite-length limit of a finite two-sided SAW and the infinite-shift limit of a one-sided infinite SAW.

\begin{theorem}\label{th:maintwosided}
The random variable $\heta$ is well-defined. Furthermore, $\eta^{m,n}$ converges weakly to~$\heta$ as $(m,n)$ tends to infinity, i.e., for all two-sided paths $\zeta$ of length $l+k$ and $\varepsilon>0$, there exist $m_0=m_0(\varepsilon,l,k),n_0=n_0(\varepsilon,l,k)$ such that for all $m\geq m_0$, $n\geq n_0$,
\begin{equation}
    |\P(\eta^{m,n}[-l,k]=\zeta)-\P(\heta[-l,k]=\zeta)|<\varepsilon.
\end{equation}
Furthermore, $T^m\eta^\infty$ converges weakly to $\heta$ as $m\rightarrow\infty$, i.e., for all $\zeta\in\SAW_k$, 
    \begin{equation}
        \limm\P((T^m\eta^\infty)[0,k]=\zeta)=\P(\heta[0,k]=\zeta).
    \end{equation}
\end{theorem}

An ingredient in several of the proofs in this paper is the following expression for the law of the one- and two-sided infinite SAW, which is also of independent interest. Let $\omega_1\in\SAW_k$ and let $\omega_2$ be a possibly infinite self-avoiding path started from 0. We say ``$\omega_2$ escapes $\omega_1$'' if $(\omega_2+\omega_1(k))\cap\omega_1=\{\omega_1(k)\}$, i.e., if the concatenation of $\omega_1$ and $\omega_2$ forms a self-avoiding path.
\begin{theorem}\label{th:dmpalt}
Let $k\in\N$ and let $\zeta,\xi$ be self-avoiding paths of lengths $k$ and $2k$ respectively. Then for all events~$E$,
\begin{equation}
	\P(\esing[0,k]=\zeta,\,(T^k\esing)[0,\infty)\in E)=\mu^{-k}\P(\esing\esc\zeta,\,\esing[0,\infty)\in E)
\end{equation}
and 
\begin{equation}
    \begin{split}
        \P&(\heta[-k,k]=\xi)=A\mu^{-2k}\cdot\\
        &\P(\esing_1[1,\infty)\cap\xi=\esing_2[1,\infty)\cap\xi=\esing_1[0,\infty)\cap\esing_2[0,\infty)=\varnothing\mid\esing_1(0)=\xi(-k),\,\esing_2(0)=\xi(k)).
    \end{split}
\end{equation}
\end{theorem}
Taking $E=\SAW_{\infty}$, we obtain the simpler expression
\begin{equation}\label{eq:dmpaltsing}
	\P(\esing[0,k]=\zeta)=\mu^{-k}\P(\esing\esc\zeta).
\end{equation}
If we interpret $\esing$ as the ``uniform random variable" on $\SAW_\infty$, this confirms the intuition that the probability that an infinite SAW starts with a certain pattern is proportional to the fraction of infinite self-avoiding paths that escape it. Similarly, the probability that a two-sided infinite SAW starts with a certain pattern is proportional to the fraction of pairs of infinite self-avoiding paths that escape either end of it it and also do not intersect each other.

It is interesting to note that the normalising constant for $k$-length paths equals $\mu^k$. As a corollary, we obtain the following expression for $\mu$. Let $e$ be a neighbour of the origin. Then
\begin{equation}
    \mu=2d\,\P(0\not\in\esing[0,\infty)\mid\esing(0)=e).
\end{equation}
More generally, we obtain
\begin{equation}
	\mu^k=c_k\P(\esing[1,\infty)\cap\eta^k[1,k]=\varnothing).
\end{equation}
This is in agreement with the observation that $c_k\sim\mu^k\P(\eta_1^k[1,k]\cap\eta_2^k[1,k]=\varnothing)^{-1}$. For the infinite SAW, we obtain equality for all $k$, rather than an asymptotic expression.

\subsubsection{Pattern theorems}
We now turn our attention to \emph{pattern theorems} for the self-avoiding walk. We use the word pattern to be synonymous to a fixed-length finite self-avoiding path. We say a pattern $\zeta\in\SAW_k$ \emph{occurs} at time $i$ in $\omega$ if~$\omega[i,i+k]=\zeta+\omega(i)$. A pattern is said to be a \emph{proper internal pattern} if there exists a self-avoiding path $\omega$ such that $\zeta$ occurs at least 3 times in $\omega$. Kesten proved that every proper internal pattern will in fact occur a positive density of times in a self-avoiding walk with very high probability \cite{kesten63number}. Let $\zeta\in\SAW_k$ and let $\omega$ be some self-avoiding path in $\Z^d$. Let 
\begin{equation}\label{eq:def-pi}
    \begin{split}
        \pi_\zeta(n,\omega)=\frac{1}{n}\sum_{i=0}^n\1_{\{\omega[i,i+k]=\zeta+\omega(i)\}}
    \end{split}
\end{equation}
be the fraction of times in $\omega$ at which the pattern $\zeta$ occurs.
\begin{theorem}[\cite{kesten63number}]\label{th:kestenpattern}
Let $d\geq1$ and let $\zeta$ be a proper internal pattern.  Then there exist $c_1,c_2,C_3>0$ such that for all $n$,
\begin{equation}
    \P(\pi_\zeta(n,\eta^n)\leq c_1)\leq C_3e^{-c_2n}.
\end{equation}
\end{theorem}
It follows that if $\zeta$ is a proper internal pattern, then $\liminf_{n\rightarrow\infty}\E[\pi_\zeta(n,\eta^n)]>0$. Since this result was published, little progress has been made on the density of proper internal patterns. It was conjectured (e.g. in \cite[Section 7.5]{madras93selfavoiding}) that $\E[\pi_\zeta(n,\eta^n)]$ should in fact converge. In this paper, we prove this for $d\geq5$ in addition to several other, stronger limit theorems.
\begin{theorem}\label{th:mainpatterns}
Let $k\in\N$ and let $\zeta\in\SAW_k$. Then as $n\rightarrow\infty$,
\begin{align}
    \E[\pi_\zeta(n,\eta^n)]\rightarrow{}&\P(\heta[0,k]=\zeta),\label{eq:conv-pattern-expectation}\\
    \pi_\zeta(n,\eta^n)\overset{\P}{\rightarrow}{}&\P(\heta[0,k]=\zeta),\label{eq:conv-pattern-prob}\\
    \pi_\zeta(n,\eta^\infty)\overset{\textrm{a.s.}}{\rightarrow}{}&\P(\heta[0,k]=\zeta).\label{eq:conv-pattern-as}
\end{align}
\end{theorem}
Note that \eqref{eq:conv-pattern-prob} already implies \eqref{eq:conv-pattern-expectation} since $\pi_\zeta$ is bounded, but we state it separately since we can prove \eqref{eq:conv-pattern-expectation} independently of \eqref{eq:conv-pattern-prob} in a much more direct way. The statement of Theorem \ref{th:mainpatterns} holds for \emph{all} patterns $\zeta$. However, if $\zeta$ is not a proper internal pattern, then $\pi_\zeta(n,\eta^n)$ converges to 0, and so $\P(\heta[0,k]=\zeta)>0$ if and only if $\zeta$ is a proper internal pattern.

If the random variables $\1_{\{\esing[i,i+k]=\zeta\}}$ were independent, then \eqref{eq:conv-pattern-as} would follow immediately from the standard strong law of large numbers. However, the SAW is highly non-Markovian and there might be strong correlations between different parts of the walk. The main ingredient in the proof is showing that the correlations between far away parts of the walk are small. As a consequence, we obtain that the infinite SAW is ergodic. We say that a random variable $X$ is \emph{ergodic} with respect to an operator $S$ if for all events $E$ that are $S$-invariant, meaning that~$\P(X\in E\triangle S(E))=0$, we have $\P(X\in E)\in\{0,1\}$. Note that we do not require $X$ to be stationary with respect to $S$.
\begin{theorem}\label{th:ergodicity}
The random variable $\esing$ is ergodic with respect to the shift operator $T$.
\end{theorem}
The strong law of large numbers \eqref{eq:conv-pattern-as} follows almost immediately from ergodicity of infinite SAW. In fact, by standard ergodic theorems, many more limit theorems follow from Theorem \ref{th:ergodicity}. In this paper, we focus on laws of large numbers for the pattern density, since that is historically one of the most studied quantities.

\subsection{Proof sketch}\label{sec:sketch-saw}
\subsubsection{Proof of Theorem \ref{th:maintwosided}}
The proof that two-sided infinite SAW is well-defined follows almost immediately from the well-known fact that the probability that two SAWs of length $n$ do not intersect equals $\frac{c_{2n}}{c_n^2}$, which converges to $A^{-1}>0$ for $d\geq5$. The fact that  two-sided finite SAW and shifted one-sided infinite SAW converge weakly to two-sided infinite SAW is intuitively obvious and the proof is straightforward. The proofs rely on a number of hitting estimates for the SAW, which can be derived from the results in Theorem \ref{th:countingwalks}. We use those same hitting estimate to prove Theorem \ref{th:dmpalt}. \eqref{eq:conv-pattern-expectation} follows immediately from the uniform convergence of finite two-sided SAW to infinite two-sided SAW.

\subsubsection{Proof of Theorem \ref{th:mainpatterns}}
The main technical achievement of the paper is the proof of Theorem \ref{th:ergodicity}, which states that infinite SAW is ergodic. To prove ergodicity, it suffices to show that two events $E_1$ and $E_2$ depending on $\esing[0,k]$ and $\esing[m,m+k]$ are asymptotically independent as $m\rightarrow\infty$. We prove this using a coupling. Given two paths $\zeta_1,\zeta_2\in\SAW_k$, and two infinite SAWs $\esing_1,\esing_2$ conditioned to start with~$\zeta_1,\zeta_2$ respectively, we construct a coupling such that almost surely, there exists $m$ for which~$T^m\esing_1[0,\infty)=T^m\esing_2[0,\infty)$. Ergodicity follows from this coupling.

The coupling relies on the \emph{domain Markov property} for infinite SAW, which essentially says that conditional on $\esing[0,k]=\zeta$, the remainder of the walk $\esing[k,\infty)$ has the distribution of an infinite SAW started from $\zeta(k)$ conditioned not to intersect $\zeta$. A formal definition of the domain Markov property is given in Theorem \ref{th:dmp}. So to sample $\esing[k,\infty)$, we could sample an independent infinite SAW $\esing_0$. If $\esing_0$ escapes $\zeta$, i.e., if $\esing_0$ started from $\zeta(k)$ does not intersect $\zeta$, set $\esing[k,\infty)=\esing_0[0,\infty)+\zeta(k)$. Otherwise, resample $\esing_0$ and repeat until it does escape $\zeta$. So one way to couple $\esing_1$ and $\esing_2$ is to use the same infinite SAW $\esing_0$ in the above sampling procedure. Then~$T^k\esing_1=T^k\esing_2=\eta_0[0,\infty)$ if $\esing_0$ escapes both $\zeta_1$ and $\zeta_2$. If $\esing_0$ escapes $\zeta_1$ but not $\zeta_2$, then set $T^k\esing_1=\eta_0[0,\infty)$ and sample $T^k\esing_2$ independently and vice versa. If $\esing_0$ does not escape both $\zeta_1$ and $\zeta_2$, then resample~$\esing_0$. Under this coupling, we have $T^k\esing_1=T^k\esing_2$ with positive probability for most $\zeta_1$ and $\zeta_2$, which is not quite good enough. The goal is to construct a coupling such that $T^m\esing_1=T^m\esing_2$ with high probability for $m\gg k$.

To that aim, we propose a very similar coupling as before. However, instead of sampling the rest of the walks $\esing_1$ and~$\esing_2$ in one go, we build them bit by bit. In the first iteration, sample an infinite SAW $\esing_{0,1}$. If $\esing_{0,1}[0,\infty)$ escapes $\zeta_1$, instead of setting $\esing_1[k,\infty)=\esing_{0,1}[0,\infty)+\zeta(k)$, set $\esing_1[k,a_{1}]=\esing_{0,1}[0,a_1-k]+\zeta(k)$ for some~$a_1\in\N$. In the same way as before, use the same walk $\esing_{0,1}$ to sample $\esing_2[k,a_1]$. So $(T^k\esing_1)[0,a_1-k]=(T^k\esing_2)[0,a_2-k]$ if $\esing_{0,1}[0,\infty)$ escapes both $\esing_1[0,k]$ and $\esing_2[0,k]$. Now let $a_2>a_1$, sample an infinite SAW $\esing_{0,2}$ and couple $\esing_1[a_1,a_2]$ and $\esing[a_1,a_2]$ in the same way. Repeat this procedure for every $\l$ and a well-chosen increasing sequence $(a_\l)_{\l\in\N}$. We say the coupling is successful at iteration $\l$ if~$T^{a_{\l-1}}\esing_1[0,a_{\l}-a_{\l-1}]=T^{a_{\l-1}}\esing_2[0,a_{\l}-a_{\l-1}]$. We refer to Section \ref{sec:coupling} for a more detailed description of the coupling. If the above was difficult to fully grasp, the diagrammatic sketch in Figure \ref{fig:coupling} might be helpful.

In dimensions $d\geq5$, two infinite SAWs do not intersect with positive probability, so at each iteration $\l$, the probability that $\esing_{0,\l}$ escapes both $\esing_1[0,a_{\l-1}]$ and $\esing_2[0,a_{\l-1}]$ (and thus that the coupling is successful) is bounded away from 0. Once the walks are successfully coupled at iteration~$\l$, then $(T^{a_{\l-1}}\esing_1)[0,a_{\l}-a_{\l-1}]=(T^{a_{\l-1}}\esing_2)[0,a_{\l}-a_{\l-1}]$. So  the coupling will be unsuccessful at the next iteration only if $\esing_{0,\l+1}$ hits $\esing_1[0,a_{\l-1}]-\esing_1(a_\l)$ but not $\esing_2[0,a_{\l-1}]-\esing_2(a_\l)$ or vice versa. If the sequence $(a_\l)_{\l}$ is chosen such that $a_{\l}$ is much bigger than~$a_{\l-1}$, then the probability of this event occurring is very small. So at each iteration, there is a positive probability of the walks coupling successfully and once the walks are coupled successfully, the probability that the couplings will remain successful at next iterations is very large. Borel-Cantelli then implies that almost surely, there exists $\l$ after which all couplings are successful, which was what we wanted to show.
 
A version of this coupling is used for finite SAWs to prove Theorem \eqref{eq:conv-pattern-prob}. To prove convergence in probability of $\pi_\zeta(n,\eta^n)$, we show that $\1_{\eta^n[i,i+k]=\zeta}$ and $\1_{\eta^n[j,j+k]=\zeta}$ are asymptotically independent as $|i-j|\rightarrow\infty$. To decorrelate finite parts of the SAW that are far apart, we couple two walks $\eta_1^n,\eta_2^n$ conditioned such that $\eta_1^n[i,i+k]=\zeta_1$ and $\eta_2^n[i,i+k]=\zeta_2$ in such a way that with high probability, $\eta_1^n$ and $\eta_2^n$ are the same far away from $[i,i+k]$. Instead of sampling a single infinite SAW at each iteration $\l$, we sample two finite SAWs $\eta_{-,\l}$ and $\eta_{+,\l}$. The coupling is successful at iteration $\l$ if appending $\eta_{-,\l}$ and $\eta_{+,\l}$ to respectively the beginning and end of $\eta_1^n[k-a_{\l},k+a_\l]$ and $\eta_2^n[k-a_{\l},k+a_\l]$ results in two self-avoiding paths. The proof that this coupling works is very similar to the infinite case.

\subsection{Discussion}\label{sec:discussion-saw}
\paragraph{Quantitative convergence} In this paper, we show several laws of large numbers, but make no attempt to quantify the speed of convergence. We also do not quantify our decorrelation results for the SAW. The greatest obstacle is that there are no known good quantitative bounds on the convergence of finite SAW to infinite SAW. If we had good bounds on the convergence of finite SAW, it should be possible to obtain polynomial concentration of $\pi_\zeta$ around its mean using the decorrelation methods from this paper. It was conjectured in \cite[Section 7.5]{madras93selfavoiding} that $\pi_\zeta$ should be exponentially concentrated. Furthermore, under polynomial concentration, it is possible to obtain a CLT for $\pi_\zeta$.

\paragraph{Self-avoiding polygons and bridges}
In this paper, we restrict our attention to self-avoiding walks. Two related random paths are \emph{self-avoiding polygons} and \emph{self-avoiding bridges}. A self-avoiding polygon of length $2n$ is a self-avoiding walk of length $2n$ conditioned to return to the origin. This can also be viewed as the union of two self-avoiding walks of length $n$ conditioned to intersect only at 0 and at its endpoints. Thus, as $n\rightarrow\infty$, one would expect the self-avoiding polygon to converge weakly to two-sided infinite SAW. However, this does not immediately follows from weak convergence of finite two-sided SAW, since the probability of the event that the two SAWs meet at its endpoints tends to 0.

The infinite self-avoiding bridge was constructed by Kesten in \cite{kesten63number} as the limit of SAW conditioned such that the first coordinate is positive at every step. The two-sided SAW should also be the infinite-shift limit of the infinite self-avoiding bridge.

\paragraph{Infinite SAW and the domain Markov property} Existence of infinite SAW for $2\leq d\leq 4$ is one of the main open problems in the theory of self-avoiding walks. Instead of showing that the limit of $\P(\eta^n[0,k]=\zeta)$ exists for every $\zeta$, we propose a different characterization of infinite SAW in Section \ref{sec:conj}. Recall the informal statement of the domain Markov property from Section \ref{sec:sketch-saw}. It seems clear that infinite SAW, if it exists, should satisfy the domain Markov property. We conjecture that there is in fact only one measure on infinite self-avoiding paths that is symmetric and satisfies the domain Markov property. Then one could simply define the law of the infinite SAW to be the unique symmetric measure satisfying the DMP. For $d\geq5$, this conjecture is true and leads to an alternative proof of existence of infinite SAW. We are not able to prove this conjecture in lower dimensions, but we give some supporting evidence. 

\paragraph{Ergodicity of other models}
The lattice tree of size $n$ is the uniform measure of connected subtrees of $\Z^d$ containing 0 of size~$n$. The weak limit of finite lattice trees has not been rigorously shown to exist. However, in high dimensions a lot is known about the number of lattice trees and the two-point function due to the lace expansion. It might be possible to show existence of the infinite lattice tree using the same methods as in this paper. In any case, we can consider subsequential limits of the finite lattice tree. The domain Markov property for the infinite lattice tree can be formulated at pivotal edges. The infinite lattice tree consists of an infinite one-sided backbone of pivotal edges with finite trees hanging off it. Each pivotal edge $e=(e_-,e_+)$ divides the tree into two disjoint clusters, one finite part containing~$0$ and~$e_-$ and one infinite part containing $e_+$. Conditional on the finite cluster being equal to $\zeta$, the rest of the tree is distributed as an infinite lattice tree started from $e_+$ conditioned not to intersect~$\zeta$. Using the exact same coupling as for the infinite SAW, it could be possible to show existence of the infinite lattice tree and show that it is ergodic with respect to the operator that shifts along the pivotal backbone.

The \emph{incipient infinite cluster} (IIC) was shown to exist in high dimensions by J\'arai and Van der Hofstad in \cite{hofstad04incipient} as the weak limit of the critical percolation cluster of 0 conditioned to be connected to $x\in\Z^d$ as~$\|x\|\rightarrow\infty$. Recently, Cabezas, Fribergh, Heydenreich and J\'arai constructed the \emph{bi-infinite} incipient cluster, which is the weak limit of the critical percolation cluster of 0 conditioned to have two disjoint paths that tend to infinity. This locally resembles the IIC as seen from a point in the backbone and is the percolation analogue of the two-sided SAW we construct in this paper. In that paper, the authors ask whether this object is ergodic with respect to the operator that shifts along pivotal edges.

The IIC construction was also recently strengthened by Chatterjee, Chinmay, Hanson and Sosoe in \cite{chatterjee2025robust}. It was shown that the IIC is also the weak limit of the cluster of 0 conditioned to be connected to some set $V_n$ and to avoid some set $D_n$ as the sets $V_n$ and $D_n$ tend to infinity, as well as the weak limit of an infinite supercritical percolation cluster as the connection probability tends to $p_c$. In that paper and in the follow-up \cite{cchs2025chemicalhighd}, the authors prove several strong decorrelation estimates for the IIC. It seems likely that these papers already contain the necessary tools to prove ergodicity with respect to pivotal shifting.

The loop-erased random walk also satisfies a variation of the domain Markov property: the remainder of a LERW conditioned to start with a path $\zeta$ is distributed as the loop-erasure of a random walk conditioned to escape $\zeta$. Ergodicity of LERW for $d\geq5$ was shown by Lawler in~\cite{lawler83connective} and for $d=4$ by the author in \cite{markering24law}. The proofs do not use the coupling technique from this paper. Several other decorrelation results have been shown in low dimensions, but most of them concern decorrelation in space, rather than time, which is what is needed for ergodicity.
 
\subsection{Outline}\label{sec:outline-saw}
We first derive some hitting estimates for SAW in Section \ref{sec:hitting} that will be needed throughout the paper. In Section \ref{sec:twosided}, we construct infinite two-sided SAW and prove Theorem \ref{th:maintwosided} and \eqref{eq:conv-pattern-expectation}. In Section \ref{sec:coupling}, we prove \eqref{eq:conv-pattern-as}. Along the way, we prove Theorem \ref{th:ergodicity} using a coupling method. In Section \ref{sec:probabilityconv}, we prove \eqref{eq:conv-pattern-prob}. In Section \ref{sec:conj} we give an alternative definition of infinite SAW and prove that it agrees with the current definition for $d\geq5$.

\section{Hitting estimates}\label{sec:hitting}
We prove several estimates on the displacement of SAW and the probability that SAW hits vertices. The most important thing to note is that the estimates are all uniform over the length of the SAW. We first state the \emph{domain Markov property} for finite SAWs. This property follows immediately from the definition of SAW and will be used throughout the paper. We will later prove an infinite-length analogue of this result.
\begin{lemma}[Finite domain Markov property]\label{lemma:dmpfin}
Let $m\geq k$ and let $\zeta,\xi$ be self-avoiding paths of length $k$ and $m-k$ respectively. Then
\begin{equation}
    \begin{split}
        &\P(\eta^m[k,m]=\xi\mid\eta^m[0,k]=\zeta)\\
        ={}&\P(\eta^{m-k}[0,m-k]=\xi\mid\eta^{m-k}[1,m-k]\cap\zeta=\varnothing,\,\eta^{m-k}(0)=\zeta(k)).
    \end{split}
\end{equation}
\end{lemma}
\begin{proof}
Note that each self-avoiding path of length $m$ that starts with $\zeta$ is the concatenation of $\zeta$ and a self-avoiding path of length $m-k$ that escapes $\zeta$. Recall the definition of~$c_n(\zeta)$ from Theorem~\ref{th:countingwalks}, which is the number of all self-avoiding paths of length $m$ that start with $\zeta$, which by the above equals the number of all self-avoiding paths of length $m-k$ starting at the tip of $\zeta$ that do not intersect $\zeta$. So
\begin{equation}
	\begin{split}
		&\P(\eta^m[k,m]=\xi\mid\eta^m[0,k]=\zeta)=\frac{c_m^{-1}}{\frac{c_m(\zeta)}{c_m}}=c_m(\zeta)^{-1}=\frac{c_{m-k}^{-1}}{\frac{c_m(\zeta)}{c_{m-k}}}\\
		={}&\P(\eta^{m-k}[0,m-k]=\xi\mid\eta^{m-k}[1,m-k]\cap\zeta=\varnothing,\,\eta^{m-k}(0)=\zeta(k)),
	\end{split}
\end{equation}
which concludes the proof.
\end{proof}

We will need the following result, which says that the SAW escapes to infinity.
\begin{lemma}\label{lemma:traveldistance}
Let $\varepsilon>0$ and $M\in\N$. Then there exists $m_0=m_0(\varepsilon,M)$ such that for all $m\geq m_0$ (including $m=\infty$),
\begin{equation}
    \P(\|\eta^m(m)\|\leq M)<\varepsilon.
\end{equation}
\end{lemma}
\begin{proof}
By \Cref{th:countingwalks}, for all $x\in\Z^d$,
\begin{equation}
    \begin{split}
        \sum_{m=0}^{\infty}\P(\eta^m(m)=x)=\sum_{m=0}^\infty\frac{c_m(x)}{c_m}\lesssim\sum_{m=0}^\infty c_m(x)\mu^{-m}<\infty.
    \end{split}
\end{equation}
Thus, $\P(\eta^m(m)=x)\rightarrow0$ as $m\to\infty$, which concludes the proof.
\end{proof}

The second lemma says that the two-point function of infinite SAW is of the same order as the two-point function of simple random walk.
\begin{lemma}\label{lemma:twopoint}
There exists $C>0$ such that for all $m$ (including $m=\infty$) and $x\in\Z^d$,
\begin{equation}
	\P(x\in\eta^m[0,m])\leq C\|x\|^{2-d}.
\end{equation}
\end{lemma}
\begin{proof}
Let $m$ be finite. Each self-avoiding path of length $m$ passing through $x$ is the concatenation of a self-avoiding path of length $n$ ending at $x$ for some $\|x\|\leq n\leq m$ and a self-avoiding path of length $m-n$. So by Theorem \ref{th:countingwalks}, there exist constants $C_1,C_2>0$ such that
\begin{equation}
	\begin{split}
		\P(x\in\eta^m[0,m])\leq&\sum_{n=0}^m\frac{c_n(x)c_{m-n}}{c_m}\\
		\leq& C_1\sum_{n=0}^{m} c_n(x)\frac{\mu^{m-n}}{\mu^m}\\
		\leq& C_1\sum_{n=0}^\infty c_n(x)\mu^{-n}\\
		\leq& C_1C_2\|x\|^{2-d}.
	\end{split}
\end{equation}
The proof is completed by choosing $C=C_1C_2$.
\end{proof}

The following lemma gives a uniform bound on the probability that the tail of a SAW hits a vertex.
\begin{lemma}\label{lemma:hittinglongtime}
Let $x\in\Z^d$ and $\varepsilon>0$. Then there exists $n_0=n_0(x,\varepsilon)$ such that for all $m\geq n\geq n_0$ (including $m=\infty$),
\begin{equation}
	\begin{split}
		\P(x\in\eta^m[n,m])\lesssim\E[\|x-\eta^n(n)\|^{2-d}]<\varepsilon.
	\end{split}
\end{equation}
\end{lemma}
\begin{proof}
Let $m$ be finite. Let $\eta_1^n$ and $\eta_2^{m-n}$ be independent SAWs starting from 0. Then by Lemmas \ref{lemma:dmpfin} and \ref{lemma:twopoint},
\begin{equation}
    \begin{split}
        \P(x\in\eta^m[n,m])=&\P(x-\eta_1^n(n)\in\eta_2^{m-n}[0,m-n]\mid\eta_1^n[1,n]\cap\eta_2^{m-n}[1,m-n]=\varnothing)\\
        \lesssim&\P(x-\eta_1^n(n)\in\eta_2^{m-n}[0,m-n])\\
        \lesssim&\E[\|x-\eta^n(n)\|^{2-d}].
    \end{split}
\end{equation}
Since $\|\eta^n(n)\|\rightarrow\infty$ in probability by Lemma \ref{lemma:traveldistance}, the last term tends to 0 as $n\rightarrow\infty$, which completes the proof. Now consider the case $m=\infty$. Then for $n\geq n_0$
\begin{equation}
    \begin{split}
        \P(x\in\eta^\infty[n,\infty))=&\lim_{k\rightarrow\infty}\P(x\in\eta^\infty[n,k])\\
        =&\lim_{k\rightarrow\infty}\limm\P(x\in\eta^m[n,k])\\
        \leq&\lim_{k\rightarrow\infty}\limm\P(x\in\eta^m[n,m])<\varepsilon.
    \end{split}
\end{equation}
\end{proof}

Lastly, we give a uniform bound on the probability that the tails of two SAWs intersect.
\begin{lemma}\label{lemma:hittinglongtime2}
Let $\varepsilon>0$. Then there exists $k_0=k_0(\varepsilon)$ such that for all $m,n\geq k\geq k_0$ (including $m,n=\infty$) and independent SAWs $\eta_1^m,\eta_2^n$, we have
\begin{equation}
    \begin{split}
        \P(\{\eta_1^m[1,k]\cap\eta_2^n[1,k]=\varnothing\}
        \setminus\{\eta_1^m[1,m]\cap\eta_2^n[1,n]=\varnothing\})<\varepsilon.
    \end{split}
\end{equation}
\end{lemma}
\begin{proof}
First note that for all $k$,
\begin{equation}\label{eq:lemmahittinglongtime2}
    \begin{split}
        &\P(\{\eta_1^m[1,k]\cap\eta_2^n[1,k]=\varnothing\}
        \setminus\{\eta_1^m[1,m]\cap\eta_2^n[1,n]=\varnothing\})\\
        \leq&\P(\eta_1^m[k,m]\cap\eta_2^n[1,n]\neq\varnothing)+\P(\eta_1^m[1,m]\cap\eta_2^n[k,n]\neq\varnothing).
    \end{split}
\end{equation}
We only bound the first term, the second term being analogous. By Lemmas \ref{lemma:twopoint} and \ref{lemma:hittinglongtime},
\begin{equation}
    \begin{split}
        \P(\eta_1^m[k,m]\cap\eta_2^n[1,n]\neq\varnothing)\leq&\sum_{x\in\Z^d}\P(x\in\eta_2^n[1,n],\,x\in\eta_1^m[k,m])\\
        \lesssim&\sum_{x\in\Z^d}\|x\|^{2-d}\E[\|x-\eta^k(k)\|^{2-d}]\\
        \lesssim&\E\left[\|\eta^k(k)\|^{4-d}\right].
    \end{split}
\end{equation}
By Lemma \ref{lemma:traveldistance}, the last quantity tends to 0 as $k\rightarrow\infty$, which completes the proof.
\end{proof}

\section{Construction of two-sided infinite SAW}\label{sec:twosided}
Recall the construction of two-sided infinite SAW in Definition \ref{def:twosidedsaw}.
In Section \ref{sec:twosidedproofs}, we show that the two-sided infinite SAW is well-defined and prove Theorem \ref{th:maintwosided}. We also prove \Cref{eq:conv-pattern-expectation} of Theorem \ref{th:mainpatterns}. We prove Theorem \ref{th:dmpalt} in Section \ref{sec:dmpaltproof}.

\subsection{Proof of Theorem \ref{th:maintwosided}}\label{sec:twosidedproofs}
Two-sided infinite SAW is well-defined by the following lemma
\begin{lemma}\label{lemma:avoidtwoinf}
Let $\eta^\infty_1,\eta^\infty_2$ be two independent infinite self-avoiding walks starting from 0. Recall the definition of the constant $A$ from Theorem \ref{th:countingwalks}. Then
\begin{equation}
    \P(\eta^\infty_1[1,\infty)\cap\eta^\infty_2[1,\infty)=\varnothing)\geq A^{-1}.
\end{equation}
\end{lemma}
\begin{proof}
Two SAWs of length $m$ starting from 0 that do not intersect form a self-avoiding path of length $2m$. So the probability that two SAWs of length $m$ starting from 0 do not intersect is $\frac{c_{2m}}{c_m^2}$. Thus, by Theorem \ref{th:countingwalks}, 
\begin{equation}
    \begin{split}
        \P(\eta^\infty_1[1,\infty)\cap\eta^\infty_2[1,\infty)=\varnothing)=&\lim_{k\rightarrow\infty}\P(\eta^\infty_1[1,k]\cap\eta^\infty_2[1,k]=\varnothing)\\
        =&\lim_{k\rightarrow\infty}\lim_{m\rightarrow\infty}\P(\eta^m_1[1,k]\cap\eta^m_2[1,k]=\varnothing)\\
        \geq&\lim_{m\rightarrow\infty}\P(\eta^m_1[1,m]\cap\eta^m_2[1,m]=\varnothing)\\
        =&\limm\frac{c_{2m}}{c_m^2}=A^{-1},
    \end{split}
\end{equation}
which concludes the proof.
\end{proof}

We start by proving the first part of Theorem \ref{th:maintwosided}.
\begin{proposition}\label{prop:twosidedconvergence}
The random variable $\eta^{m,n}$ converges weakly to $\heta$ uniformly in $m$ and $n$, i.e., for all $\zeta$ of length $l+k$ and $\varepsilon>0$, there exist $M=M(\varepsilon,l,k),N=N(\varepsilon,l,k)$ such that for all $m\geq M$, $n\geq N$,
\begin{equation}
    |\P(\eta^{m,n}[-k,l]=\zeta)-\P(\heta[-k,l]=\zeta)|<\varepsilon.
\end{equation}
\end{proposition}
\begin{proof}
Let $\varepsilon>0$ and $t\in\N$. By weak convergence of finite SAW to one-sided infinite SAW, there exist $M=M(\varepsilon,k)$ and $N=N(\varepsilon,k)$ such that for all $m\geq M$ and $N\geq n$ and all events $A$ depending on $t$ coordinates, we have
\begin{equation}
    \begin{split}
        |\P(\eta^m,\eta^n\in A)-\P(\eta_1^\infty,\eta_2^\infty\in A)|<\varepsilon/3.
    \end{split}
\end{equation}
This is due to weak convergence of finite SAW to infinite SAW. Let $\zeta_-:=\zeta[-k,0]$ and $\zeta_+:=\zeta[0,l]$. Then for all $t$ large enough and $m\geq M(t,\varepsilon)$ and $n\geq N(t,\varepsilon)$, we have
\begin{equation}
    \begin{split}
        &|\P(\eta^{m,n}[-k,l]=\zeta)-\P(\heta[-k,l]=\zeta)|\\
        =&|\P(\eta^m[0,k]=\zeta_-,\,\eta^n[0,l]=\zeta_+\mid\eta^m[1,m]\cap\eta^n[1,n]=\varnothing)-\P(\heta[-k,l]=\zeta)|\\
        <&|\P(\eta^m[0,k]=\zeta_-,\,\eta^n[0,l]=\zeta_+\mid\eta^m[1,t]\cap\eta^n[1,t]=\varnothing)-\P(\heta[-k,l]=\zeta)|+\frac{\varepsilon}{3}\\
        <&|\P(\eta_1^\infty[0,k]=\zeta_-,\,\eta_2^\infty[0,l]=\zeta_+\mid\eta_1^\infty[1,t]\cap\eta_2^\infty[1,t]=\varnothing)-\P(\heta[-k,l]=\zeta)|+\frac{2\varepsilon}{3}\\
        <&|\P(\eta_1^\infty[0,k]=\zeta_-,\,\eta_2^\infty[0,l]=\zeta_+\mid\eta_1^\infty[1,\infty)\cap\eta_2^\infty[1,\infty)=\varnothing)-\P(\heta[-k,l]=\zeta)|+\varepsilon\\
        =&\varepsilon.
    \end{split}
\end{equation}
For the first inequality, we use Lemma \ref{lemma:hittinglongtime2} combined with the fact that $\P(\eta^m[1,t]\cap\eta^n[1,t]=\varnothing)\geq\frac{c_{m+n}}{c_mc_n}$ is bounded from below. The last inequality follows in the same way. The last equality follows from the definition of $\heta$.
\end{proof}

We also prove the second part of Theorem \ref{th:maintwosided}.
\begin{proposition}
For all $k$ and $\zeta\in\SAW_k$,
    \begin{equation}
        \limm\P(T^m\eta^\infty[0,k]=\zeta)=\P(\heta[0,k]=\zeta).
    \end{equation}
\end{proposition}
\begin{proof}
Let $\varepsilon>0$. Then by Proposition \ref{prop:twosidedconvergence}, there exists $M=M(k,\varepsilon)$ such that for all $m,n\geq M$, we have $|\P(\heta[0,k]=\zeta)-\P(\eta^{m,n}[0,k]=\zeta)|<\varepsilon$. Then by definition of the infinite one-sided SAW, there exists $N=N(m+k,\varepsilon)\geq k$ such that for all $n\geq N$,
\begin{equation}
    \begin{split}
    	&|\P(\heta[0,k]=\zeta)-\P(T^m\eta^\infty[0,k]=\zeta)|\\
    	<&
        |\P(\eta^{m,n}[0,k]=\zeta)-\P(T^m\eta^\infty[0,k]=\zeta)|+\varepsilon\\
        =&|\P(\eta^{m+n}[m,m+k]-\eta^{m+n}(m)=\zeta)-\P(T^m\eta^\infty[0,k]=\zeta)|+\varepsilon\\
        <&|\P(\eta^\infty[m,m+k]-\esing(m)=\zeta)-\P(T^m\eta^\infty[0,k]=\zeta)|+2\varepsilon\\
        =&2\varepsilon.
    \end{split}
\end{equation}
Since $\varepsilon>0$ was arbitrary, the proof is complete.
\end{proof}

We are now ready to prove \eqref{eq:conv-pattern-expectation}.
\begin{proof}[Proof of \Cref{eq:conv-pattern-expectation} of Theorem \ref{th:mainpatterns}]
By the uniform convergence of two-sided SAW in Proposition \ref{prop:twosidedconvergence} and the domain Markov property in Lemma \ref{lemma:dmpfin},
\begin{equation}
    \begin{split}
        \limn\E[\pi_\zeta(n,\eta^n)]=&\limn\frac{1}{n}\sum_{i=0}^{n}\P(\eta^n[i,i+k]=\zeta)\\
        =&\limn\frac{1}{n}\sum_{i=0}^{n}\P(\eta^{n-i}[0,k]=\zeta\mid\eta^{n-i}[1,n-i]\cap\eta^i[0,i]=\varnothing)\\
        =&\limn\frac{1}{n}\sum_{i=0}^{n}\P(\eta^{n-i,i}[0,k]=\zeta)\\
        =&\P(\heta[0,k]=\zeta),
    \end{split}
\end{equation}
which was what we wanted to show.
\end{proof}

\subsection{Proof of Theorem \ref{th:dmpalt}}\label{sec:dmpaltproof}
\begin{proof}[Proof of Theorem \ref{th:dmpalt}]
We first prove the one-sided part. Note that it suffices to prove that for all~$l$ and all $\zeta'\in\SAW_l$,
\begin{equation}
	\P(\esing[0,k]=\zeta,\,T^k\esing[0,l]=\zeta')=\mu^{-k}\P(\esing\esc\zeta,\,\esing[0,l]=\zeta').
\end{equation}
Let $n>k+l$. Recall that $c_n(\zeta)$ denotes the number of $n$-step self-avoiding paths which start with $\zeta$. Denote by~$\zeta\oplus\zeta'$\phantomsection\label{def:oplus} the concatenation of $\zeta$ and $\zeta'$, i.e., $\zeta\oplus\zeta':=(\zeta(0),\ldots,\zeta(k),\zeta'(1)+\zeta(k),\ldots,\zeta'(l)+\zeta(k))$. Note that~$c_n(\zeta\oplus\zeta')$ equals the number of $n-k$-step paths that escape $\zeta$ and start with $\zeta'$. Thus,
\begin{equation}
    \begin{split}
       & \P(\eta^n[0,k]=\zeta,\,(T^k\eta^n)[0,l]=\zeta')=\frac{c_n(\zeta\oplus\zeta')}{c_n}=\frac{c_{n-k}}{c_n}\cdot\frac{c_n(\zeta\oplus\zeta')}{c_{n-k}}\\
        ={}&\frac{c_{n-k}}{c_n}\P(\eta^{n-k}[1,n-k]\cap\zeta=\varnothing,\,\eta^{n-k}[0,l]=\zeta'\mid\eta^{n-k}(0)=\zeta(k)).
    \end{split}
\end{equation}
Taking the limit $n\rightarrow\infty$, we obtain that the left-hand side converges to $\P(\esing[0,k]=\zeta,T^k\esing[0,l]=\zeta')$, and $\frac{c_{n-k}}{c_n}$ converges to $\mu^{-k}$. So it remains to show that
\begin{equation}
	\begin{split}
		&\limn\P(\eta^{n-k}[1,n-k]\cap\zeta=\varnothing,\,\eta^{n-k}[0,l]=\zeta'\mid\eta^{n-k}(0)=\zeta(k))\\
		=&\P(\esing[1,\infty)\cap\zeta=\varnothing,\,\esing[0,l]=\zeta'\mid\esing(0)=\zeta(k)).
	\end{split}
\end{equation}
By a union bound and Lemma \ref{lemma:hittinglongtime}, there exists $m_0=m_0(k,\varepsilon)$ such that for all $m\geq m_0$, and all $n-k\geq m$ (including $n-k=\infty$),
\begin{equation}
    \P(\eta^{n-k}[m,n-k]\cap\zeta\neq\varnothing\mid\eta^{n-k}(0)=\zeta(k))<\varepsilon.
\end{equation}
So for $n-k$ sufficiently large and $m\geq m_0$,
\begin{equation}
    \begin{split}
        &|\P(\eta^{n-k}[1,n-k]\cap\zeta=\varnothing,\,\eta^{n-k}[0,l]=\zeta'\mid\eta^{n-k}(0)=\zeta(k))\\
        &-\P(\esing[1,\infty)\cap\zeta=\varnothing,\,\eta^{\infty}[0,l]=\zeta'\mid\esing(0)=\zeta(k))|\\
        \leq&|\P(\eta^{n-k}[1,m]\cap\zeta=\varnothing,\,\eta^{n-k}[0,l]=\zeta'\mid\eta^{n-k}(0)=\zeta(k))\\
        &-\P(\esing[1,m]\cap\zeta=\varnothing,\,\eta^{\infty}[0,l]=\zeta'\mid\esing(0)=\zeta(k))|+\varepsilon\\
        \leq&2\varepsilon.
    \end{split}
\end{equation}
The last step follows from weak convergence of finite SAW to infinite SAW. This completes the proof since $\varepsilon$ was arbitrary.

The proof for two-sided SAW is similar. We have
\begin{equation}
    \begin{split}
        &\P(\eta^{m,n}[-k,k]=\xi)\\
        =&\frac{c_{m-k}c_{n-k}}{c_mc_n}\cdot\frac{\P(\eta_1^{m-k}\cup\xi\cup\eta_2^{n-k}\text{ is a self-avoiding path}\mid\eta_1^{m-k}(0)=\xi(-k),\,\eta_2^{n-k}(0)=\xi(k))}{\P(\eta_1^{m}\cup\eta_2^{n}\text{ is a self-avoiding path}\mid\eta_1^{m}(0)=0=\eta_2^{n}(0))}.
    \end{split}
\end{equation}
Let $\varepsilon>0$. By Lemmas \ref{lemma:hittinglongtime} and \ref{lemma:hittinglongtime2}, there exists $l_0$ such that for all $l\geq l_0$ and $m-k,n-k\geq l$, we have
\begin{equation}
    \begin{split}
        \P(\eta^{m-k}_1[l,m-k],\eta^{n-k}_2[l,m-k],\xi\text{ are not disjoint}\mid\eta_1^{m-k}(0)=\xi(-k),\,\eta_2^{n-k}(0)=\xi(k))<\varepsilon.
    \end{split}
\end{equation}
The proof now follows in the same way as for the one-sided infinite SAW, noting that by Theorem \ref{th:maintwosided}, the left-hand side converges to $\P(\heta[-k,k]=\xi)$ and that $\P(\eta_1^{m}\cup\eta_2^{n}\text{ is a self-avoiding path}\mid\eta_1^{m}(0)=0=\eta_2^{n}(0))\rightarrow A^{-1}$.
\end{proof}

\section{Ergodicity}\label{sec:coupling}
In this section we prove Theorem \ref{th:ergodicity}, from which \eqref{eq:conv-pattern-as} follows. Recall the definition of the shift operator $T$ from Section \ref{sec:intro}. Ergodicity of one-sided infinite SAW with respect to $T$ follows from the following proposition. 

\begin{proposition}\label{prop:ergodicity}
Let $k\in\N$ and let $\esing_1$ and $\esing_2$ be infinite self-avoiding walks conditioned to start with paths $\zeta_1,\zeta_2\in\SAW_k$ respectively. Let $\|\cdot\|_{\TV}$ be the total variation distance. Then
\begin{equation}
	\limm\|T^m\esing_1[0,\infty)-T^m\esing_2[0,\infty)\|_{\TV}=0.
\end{equation}
\end{proposition}
For the proof, we construct a coupling of two infinite SAWs conditioned to start with two different patterns such that the two walks agree after a certain time with high probability. As a consequence of Proposition \ref{prop:ergodicity}, we obtain that parts of the self-avoiding walk that are far away are asymptotically independent.
\begin{corollary}\label{cor:ergodicity}
Let $\zeta_1,\zeta_2\in\SAW_k$. Then
\begin{equation}
    \lim_{m\rightarrow\infty}\P(\eta^\infty[0,k]=\zeta_1,\,\eta^\infty[m,m+k]=\zeta_2)=\P(\eta[0,k]=\zeta_1)\P(\heta[0,k]=\zeta_2).
\end{equation}
\end{corollary}
\begin{proof}
Let $\zeta_1'\in\SAW_k$. Then by Proposition \ref{prop:ergodicity},
\begin{equation}
    \begin{split}
        \limm|\P(\eta^\infty[m,m+k]=\zeta_2\mid\eta^\infty[0,k]=\zeta_1)-\P(\eta^\infty[m,m+k]=\zeta_2\mid\eta^\infty[0,k]=\zeta'_1)|=0.
    \end{split}
\end{equation}
So by Theorem \ref{th:maintwosided},
\begin{equation}
    \begin{split}
        \P(\heta[0,k]=\zeta_2)=&\limm\P(\eta^\infty[m,m+k]=\zeta_2)\\
    =&\sum_{\zeta_1'\in\SAW_k}\limm\P(\eta^\infty[m,m+k]=\zeta_2\mid\eta^\infty[0,k]=\zeta'_1)\P(\eta^\infty[0,k]=\zeta'_1)\\
    =&\sum_{\zeta_1'\in\SAW_k}\limm\P(\eta^\infty[m,m+k]=\zeta_2\mid\eta^\infty[0,k]=\zeta_1)\P(\eta^\infty[0,k]=\zeta'_1)\\
    =&\limm\P(\eta^\infty[m,m+k]=\zeta_2\mid\eta^\infty[0,k]=\zeta_1)\\
    =&\limm\P(\eta^\infty[m,m+k]=\zeta_2,\eta^\infty[0,k]=\zeta_1)\P(\eta^\infty[0,k]=\zeta_1)^{-1},
    \end{split}
\end{equation}
which is what we want to show.
\end{proof}

\begin{proof}[Proof of Theorem \ref{th:ergodicity}]
Ergodicity of $\esing$ follows directly from Corollary \ref{cor:ergodicity} by standard results in ergodic theory (see \cite[Lemma 6.7.4]{gray09probability}).
\end{proof}

\begin{proof}[Proof of \Cref{eq:conv-pattern-as} of Theorem \ref{th:mainpatterns}]
It follows from weak convergence of $T^m\eta^\infty$ to $\heta$ that $\eta^\infty$ is asymptotic mean stationary with limit $\heta$. It then follows from Corollary \ref{cor:ergodicity} that $\heta$ is ergodic with respect to the sigma-algebra generated by cylinder sets. A standard ergodic theorem (see \cite[Corollary 7.2.1]{gray09probability}) immediately implies the desired strong law of large numbers.
\end{proof}

In Section \ref{sec:constructioncoupling}, we construct the coupling and prove that it is indeed a coupling using the domain Markov property. In Section \ref{sec:ergodicityproof}, we prove Proposition \ref{prop:ergodicity} using the coupling.

\subsection{Coupling}\label{sec:constructioncoupling}
Let $\omega_1\in\SAW_k$ and let $\omega_2$ be a possibly infinite self-avoiding path. We denote by $\omega_1\oplus\omega_2$ the concatenation of $\omega_1$ and $\omega_2$, i.e., $\omega_1\oplus\omega_2:=(\omega_1(0),\ldots,\omega_1(k),\omega_2(1)+\omega_1(k),\omega_2(2)+\omega_1(k),\ldots)$. Recall that we say $\omega_2$ escapes $\omega_1$ if $\omega_1\oplus\omega_2$ forms a self-avoiding path, i.e., if $\omega_2[1,\infty)\cap(\omega_1[0,k]-\omega_1(k))=\varnothing$.

Recall the definition of the left-shift operator $T$ from Section \ref{sec:intro}. Let $\esing_1$ and $\esing_2$ be as in Proposition \ref{prop:ergodicity}. We construct a coupling such that $T^m\esing_1=T^m\esing_2$ with probability tending to~1 as $m\rightarrow\infty$. Let $(a_\l)_{\l\geq0}$ be an increasing sequence of positive integers with $a_0=k$. The other values of $a_\l$ will be specified later. At each iteration $\l$, given~$\esing_1[0,a_\l]$ and~$\esing_2[0,a_\l]$, we construct a coupling of $\esing_1[a_\l,a_{\l+1}]$ and $\esing_2[a_\l,a_{\l+1}]$ as follows.

At each iteration $\l$, we sample an independent infinite self-avoiding walk $\esing_{0,\l}$ starting from 0. If this walk escapes both $\esing_1[0,a_\l]$ and $\esing_2[0,a_\l]$, then set $\esing_1[0,a_{\l+1}]:=\esing_1[0,a_{\l}]\oplus\esing_{0,\l}[0,a_{\l+1}-a_\l]$ and $\esing_2[0,a_{\l+1}]:=\esing_2[0,a_\l]\oplus\esing_{0,\l}[0,a_{\l+1}-a_\l]$. If this walk escapes $\esing_1[0,a_\l]$ but does not escape~$\esing_2[0,a_\l]$, then set~$\esing_1[0,a_{\l+1}]:=\esing_1[0,a_{\l}]\oplus\esing_{0,\l}[0,a_{\l+1}-a_\l]$ and sample $\esing_2[a_\l,a_{\l+1}]$ independently. Similarly, if the walk escapes $\esing_2[0,a_\l]$ but does not escape $\esing_1[0,a_\l]$, then set~$\esing_2[0,a_{\l+1}]:=\esing_2[0,a_{\l}]\oplus\esing_{0,\l}[0,a_{\l+1}-a_\l]$ and sample $\esing_1[a_\l,a_{\l+1}]$ independently. If the walk does not escape $\esing_1[0,a_\l]$ and does not escape $\eta_2[0,a_\l]$, resample $\esing_{0,\l}$ and continue until it escapes one of the two paths. We refer to Figure \ref{fig:coupling} for an illustration of the first iteration of the coupling.

\begin{figure}
    \centering
    \begin{subfigure}{\textwidth}
        \begin{subfigure}{.5\textwidth}
            \centering
            \includestandalone{TikZsuccessbefore}
        \end{subfigure}
        \begin{subfigure}{.5\textwidth}
            \centering
            \includestandalone{TikZsuccessafter}
        \end{subfigure}
        \caption{A successful coupling in the first iteration. The walk $\esing_{0,1}$ escapes both $\esing_1[0,k]$ and $\esing_2[0,k]$. Thus, $(T^{k}\esing_1)[0,a_1-k]=(T^{k}\esing_2)[0,a_1-k]=\esing_{0,1}[0,a_1-k]$.}
    \end{subfigure}
    \begin{subfigure}{\textwidth}
        \begin{subfigure}{.5\textwidth}
            \centering
            \includestandalone{TikZfailbefore}
        \end{subfigure}
        \begin{subfigure}{.5\textwidth}
            \centering
            \includestandalone{TikZfailafter}
        \end{subfigure}
        \caption{An unsuccessful coupling in the first iteration. The walk $\esing_{0,1}$ escapes $\esing_1[0,k]$, but not $\esing_2[0,k]$. Thus, $(T^{k}\esing_1)[0,a_1-k]=\esing_{0,1}[0,a_1-k]$ and $\esing_2[k,a_1]$ is sampled independently.}
    \end{subfigure}
    \caption{A sketch of the first iteration of the coupling. The solid blue and red lines represent $\esing_1$ and $\esing_2$ respectively. The dashed olive line represents $\esing_{0,1}$. The top pictures are an example of a successful coupling, in the bottom pictures the coupling is not successful.}
    \label{fig:coupling}
\end{figure}

The proof that this is in fact a coupling relies on the following \emph{domain Markov property} of infinite SAW. The property says that a SAW conditioned to start with a certain pattern $\zeta$ is distributed as the concatenation of $\zeta$ and a SAW conditioned to escape $\zeta$. For finite SAWs, this property is immediate and was already proved in \ref{lemma:dmpfin}. The statement for infinite SAWs follows from Theorem~\ref{th:dmpalt}.

\begin{theorem}[Infinite domain Markov property]\label{th:dmp}
For every $k\in\N$, $\zeta\in\SAW_k$ and every event $E$ (in the sigma-algebra on $\SAW_\infty$ generated by cylinder events),
\begin{equation}
    \begin{split}
        \P(\eta^\infty[k,\infty)\in E\mid\esing[0,k]=\zeta)= \P(\eta^\infty[0,\infty)\in E\mid\eta^\infty[1,\infty)\cap\zeta=\varnothing,\,\eta^{\infty}(0)=\zeta(k))
    \end{split}
\end{equation}
\end{theorem}
\begin{proof}
We have
\begin{equation}
	\begin{split}
		&\P(\esing[k,\infty)\in E\mid\esing[0,k]=\zeta)\\
		=&\P(\esing[k,\infty)\in E\,\,\esing[0,k]=\zeta)\P(\esing[0,k]=\zeta)^{-1}\\
		=&\mu^{-k}\P(\esing[0,\infty)\in E,\,\esing[1,\infty\cap\zeta=\varnothing\mid\esing(0)=\zeta(k))\P(\esing[0,k]=\zeta)^{-1}\\
		=&\P(\esing[0,\infty)\in E\mid\esing[1,\infty)\cap\zeta=\varnothing,\esing(0)=\zeta(k))\mu^{-k}\P(\esing\esc\zeta)\P(\esing[0,k]=\zeta)^{-1}\\
		=&\P(\esing[0,\infty)\in E\mid\esing[1,\infty)\cap\zeta=\varnothing,\esing(0)=\zeta(k)),
	\end{split}
\end{equation}
which is what we wanted to show.
\end{proof}

We are now ready to prove that the construction above is a coupling. The proof is by induction. The base case is clearly true, since $\esing_1[0,a_0]=\esing_1[0,k]=\zeta_1$ and $\esing_2[0,a_0]=\esing_2[0,k]=\zeta_2$. Now assume that for some $\l\geq0$, $\esing_1[0,a_\l]$ and $\esing_2[0,a_\l]$ are distributed as the first $a_\l$ steps of an infinite SAW conditioned to start with $\zeta_1$ and $\zeta_2$ respectively. From the construction, it is clear that~$\esing_1[a_\l,a_{\l+1}]$ is distributed as the first $a_{\l+1}-a_{\l}$ steps of an infinite self-avoiding walk conditioned to escape $\esing_1[0,a_\l]$. It follows from the domain Markov property Theorem \ref{th:dmp} that~$\esing_1[0,a_{\l+1}]$ is indeed distributed as the first $a_{\l+1}$ steps of an infinite SAW conditioned to start with~$\esing_1[0,a_\l]$. By the induction hypothesis, $\esing_1[0,a_\l]$ is distributed as the first $a_\l$ steps of an infinite SAW conditioned to start with $\zeta_1$. So $\esing_1[0,a_{\l+1}]$ is distributed as the first $a_{\l+1}$ steps of an infinite SAW conditioned to start with $\zeta_1$. The proof for $\esing_2$ is analogous.

\subsection{Proof of ergodicity}\label{sec:ergodicityproof}
We now prove that with probability 1, there exists $\l$ such that $T^{a_\l}\esing_1=T^{a_\l}\esing_2$ for a suitably chosen sequence $(a_\l)_{\l\geq0}$, which implies Proposition \ref{prop:ergodicity}. We say $\esing_1$ and $\esing_2$ are successfully coupled at iteration $\l$ if $T^{a_{\l-1}}\esing_1[0,a_\l-a_{\l-1}]=T^{a_{\l-1}}\esing_2[0,a_\l-a_{\l-1}]$. We prove that if the coupling is not successful at some iteration, then the probability that the coupling is succesful at the next iteration is bounded away from 0. Furthermore, we show that once $\esing_1$ and $\esing_2$ are successfully coupled, the probability that the coupling is unsuccessful at the next iteration is exponentially small. This implies that the probability of the two walks not being successfully coupled at iteration $\l$ decays exponentially in $\l$. So by Borel-Cantelli, a.s. there are only finitely many iterations where $\esing_1$ and~$\esing_2$ are not successfully coupled.

\subsubsection{Probability of coupling successfully}
Recall from Lemma \ref{lemma:avoidtwoinf} for $d\geq5$, the probability that two infinite SAWs do not intersect is positive. It follows that the escape probability of an infinite SAW and the first steps of another infinite SAW is positive. At each iteration $\l$, the probability that the coupling is successful equals the probability that an independent infinite SAW escapes the first $a_{\l-1}$ steps of both $\esing_1$ and $\esing_2$. However, on the event that the coupling was unsuccessful at iteration $\l-1$, $\esing_1$ and $\esing_2$ are not distributed as independent standard SAWs, but rather distributed as SAWs conditioned to start with a certain pattern and to intersect each other. So we need to do some work to show that the probability of successfully coupling at each iteration is bounded from below.

We first prove the following key lemma on the probability of an infinite self-avoiding walk $\eta_1'$ escaping the first $m$ steps of an infinite self-avoiding walk $\eta_2'$ conditioned to start with some path of length $k$. The lemma says that if $m$ is sufficiently large, then this probability is at least the probability that two unconditioned self-avoiding walks do not intersect. So in a sense, the walk $\eta_2'$ forgets its initial conditioning after a large number of steps, which is what it means for the walk to be ergodic.

\begin{lemma}\label{lemma:avoidingconditioned}
Let $\eta'_1,\eta'_2$ be independent infinite self-avoiding walks. For all $\xi\in\SAW_k$,
\begin{equation}
    \liminf_{m\rightarrow\infty}\P(\eta'_1\esc\eta_2'[0,m]\mid\eta'_2[0,k]=\xi)\geq A^{-1},
\end{equation}
with $A$ as in Theorem \ref{th:countingwalks}. 
\end{lemma}
\begin{proof}
Write $\zeta\succeq\xi$ if $\zeta$ is a self-avoiding path of length $m\geq k$ that starts with $\xi$. Recall the notation~$c_m(\xi)$, which is the number of $m$-step self-avoiding paths that start with $\xi$. Then, by~\eqref{eq:dmpaltsing}, we have
\begin{equation}
    \begin{split}
        &\liminf_{m\rightarrow\infty}\P(\eta'_1\esc\eta_2'[0,m]\mid\eta'_2[0,k]=\xi)\\
        =&\liminf_{m\rightarrow\infty}\sum_{\zeta\succeq\xi}\P(\esing\esc\zeta)\P(\esing[0,m]=\zeta)\P(\esing[0,k]=\xi)^{-1}\\
        =&\liminf_{m\rightarrow\infty}\mu^m\P(\esing[0,k]=\xi)^{-1}\sum_{\zeta\succeq\xi}\P(\esing[0,m]=\zeta)^{2}\\
        \geq&\liminf_{m\rightarrow\infty}\mu^m\P(\esing[0,k]=\xi)^{-1}c_m(\xi)\left(\sum_{\zeta\succeq\xi}\P(\esing[0,m]=\zeta)c_m(\xi)^{-1}\right)^2\\
        =&\liminf_{m\rightarrow\infty}\mu^m\P(\esing[0,k]=\xi)^{-1}c_m(\xi)\P(\esing[0,k]=\xi)^2c_m(\xi)^{-2}\\
        =&\liminf_{m\rightarrow\infty}\frac{\mu^m\P(\esing[0,k]=\xi)}{c_m\P(\eta^m[0,k]=\xi)}
        =A^{-1}
    \end{split}
\end{equation}
The inequality follows from Jensen. The last inequality follows from weak convergence of finite SAW to infinite SAW and the fact that $c_n\sim A\mu^n$.
\end{proof}

We now show that the probability of successfully coupling at iteration $\l+1$ given that the walks are not successfully coupled at iteration $\l$ is bounded from below uniformly in $\l$ if $a_\l$ grows fast enough. For the proof, we show that with high probability, if the coupling was unsuccessful at iteration $\l$, then $\esing_1$ and $\esing_2$ intersect each other in the first few steps after $a_{\l-1}$. Then on this event, by Lemma \ref{lemma:avoidingconditioned}, we can choose $a_{\l}$ to be large enough that $\esing_{0,\l+1}$ escapes both $\esing_1[0,a_\l]$ and~$\esing_2[0,a_\l]$ with positive probability.

\begin{lemma}\label{lemma:probcouple}
Let $\esing_1$ and  $\esing_2$ be infinite SAWs conditioned to start with $\zeta_1$ and $\zeta_2$ of length $k$ respectively, coupled as above. Then there exists $c>0$ and a sequence $(a_\l)_{\l\geq0}$ such that for all $\l\geq2$,
\begin{equation}
    \P(\text{coupling is successful at iteration $\l+1$}\mid \text{coupling is not successful at iteration $\l$})>c.
\end{equation}
\end{lemma}
\begin{proof}
If $\esing_1$ and $\esing_2$ are not successfully coupled at iteration~$\l$, then it means that $\esing_{0,\l}$ at iteration~$\l$ escaped $\esing_2[0,a_{\l-1}]$ but not $\esing_1[0,a_{\l-1}]$ or the other way around. Hence,
\begin{equation}
    \begin{split}
        &\P(\text{coupling is successful at iteration $\l+1$}\mid \text{coupling is not successful at iteration $\l$})\\
        \geq&\min_{\zeta_1,\zeta_2}\P(\esing_{0,\l+1}\esc\esing_1[0,a_\l]\text{ and }\esing_2[0,a_\l]\\
        &\mid\esing_1[0,a_{\l-1}]=\zeta_1,\,\esing_2[0,a_{\l-1}]=\zeta_2,\,\esing_{0,\l}\esc\zeta_1,\,\esing_{0,\l}\nesc\zeta_2),
        \end{split}
\end{equation}
where the minimum is taken over all paths $\zeta_1,\zeta_2\in\SAW_{a_{\l-1}}$ such that the conditioning event has positive probability. Note that by the DMP, on the above conditioning, $\esing_1$ has the distribution of an infinite SAW $\esing$ conditioned on the event $\{\esing[0,a_{\l-1}]=\zeta_1,\,\esing[a_{\l-1},\infty)-\zeta_1(a_{\l-1})\nesc\zeta_2\}$. Thus, the above becomes
\begin{equation}\label{eq:secondstepsuccess}
	\begin{split}
		&\min_{\zeta_1,\zeta_2}\P(\esing_{0,\l+1}\esc\esing[0,a_\l]\text{ and }\esing_2[0,a_\l]\\
        &\mid\esing[0,a_{\l-1}]=\zeta_1,\,\esing_2[0,a_{\l-1}]=\zeta_2,\esing[a_{\l-1},\infty)-\zeta_1(a_{\l-1})\nesc\zeta_2)\\
        \geq&\min_{\zeta_1,\zeta_2}\P(\esing_{0,\l+1}\esc\esing[0,a_\l]\text{ and }\esing_2[0,a_\l]\\
        &\mid\esing[0,a_{\l-1}]=\zeta_1,\,\esing_2[0,a_{\l-1}]=\zeta_2,\,\esing[a_{\l-1},\tilde{a}_\l]-\zeta_1(a_{\l-1})\nesc\zeta_2)\\
        &\cdot\P(\esing[a_{\l-1},\tilde{a}_\l]-\zeta(a_{\l-1})\nesc\zeta_2\\
        &\mid\esing[0,a_{\l-1}]=\zeta_1,\,\eta_2[0,a_{\l-1}]=\zeta_2,\,\esing[a_{\l-1},\infty)-\zeta_1(a_{\l-1})\nesc\zeta_2)
   \end{split}
\end{equation}
for some $a_{\l-1}\leq\tilde{a}_\l\leq a_{\l}$ to be specified later. Note that for all paths $\zeta_1,\zeta_2$ of length $a_{\l-1}$,
\begin{equation}
    \begin{split}
        1=&\P(\esing_{0,\l}[a_{\l-1},\infty)-\zeta_1(a_{\l-1})\nesc\zeta_2\\
        &\mid\esing[0,a_{\l-1}]=\zeta_1,\,\eta_2[0,a_{\l-1}]=\zeta_2,\,\esing_{0,\l}[a_{\l-1},\infty)-\zeta_1(a_{\l-1})\nesc\zeta_2)\\
        =&\lim_{n\rightarrow\infty}\P(\esing_{0,\l}[a_{\l-1},n]-\zeta_1(a_{\l-1})\nesc\zeta_2\\
        &\mid\esing[0,a_{\l-1}]=\zeta_1,\,\eta_2[0,a_{\l-1}]=\zeta_2,\,\esing_{0,\l}[a_{\l-1},\infty)-\zeta_1(a_{\l-1})\nesc\zeta_2).
    \end{split}
\end{equation}
So for $\tilde{a}_\l$ large enough, we have
\begin{equation}
    \begin{split}
        &\min_{\zeta_1,\zeta_2}\P(\esing[a_{\l-1},\tilde{a}_\l]-\zeta(a_{\l-1})\nesc\zeta_2\\
        &\mid\esing[0,a_{\l-1}]=\zeta_1,\,\eta_2[0,a_{\l-1}]=\zeta_2,\,\esing[a_{\l-1},\infty)-\zeta_1(a_{\l-1})\nesc\zeta_2)\geq\frac{1}{2}.
    \end{split}
\end{equation}
Thus, \eqref{eq:secondstepsuccess} is bounded from below by
\begin{equation}
	\begin{split}
        &\frac{1}{2}[2\min_{\zeta_1,\zeta_2}\P(\esing_{0,\l+1}\esc\esing[0,a_\l]\\
        &\mid\esing[0,a_{\l-1}]=\zeta_1,\,\esing[a_{\l-1},\tilde{a}_\l]-\zeta_1(a_{\l-1})\nesc\zeta_2)-1].
	\end{split}
\end{equation}
Note that the conditioning in the last probability above depends only on the first $\Tilde{a}_{\l}$ steps. So by Lemma \ref{lemma:avoidingconditioned}, there exists $m_0=m_0(\Tilde{a}_\l,\varepsilon)$ such that if we choose $a_\l>m_0$, then the last line is bounded from below by $A^{-1}(1-\varepsilon)-\frac{1}{2}$. It was shown in \cite{hara92selfavoiding} that for $d\geq 5$, $A<2$, so for $\varepsilon$ small enough $A^{-1}(1-\varepsilon)-\frac{1}{2}>0$, which concludes the proof.
\end{proof}

\subsubsection{Probability of staying successfully coupled}\label{sec:stayingcoupled}
We now show that the probability that two successfully coupled walks uncouple is exponentially small in $\l$, if we choose $a_\l$ to grow sufficiently fast. If the two walks are successfully coupled at iteration $\l$, then they disagree on at most the interval $[0,a_{\l-1}]$, which is relatively small compared to $[0,a_\l]$ if $a_\l$ grows exponentially. So the probability of $\esing_{0,\l}$ hitting one walk but not the other is small, which means the probability of uncoupling is small.
\begin{lemma}\label{lemma:probuncouple}
Let $\esing_1,\esing_2$ be infinite SAWs conditioned to start with $\zeta_1$ and $\zeta_2$ respectively, coupled as above. Then there exists a sequence $(a_{\l})_{\l}$ and $C=C(\zeta_1,\zeta_2)$ such that for all $\l$,
\begin{equation}
    \P(\text{coupling is not successful at iteration $\l+1$},\, \text{coupling is successful at iteration $\l$})\leq Ce^{-\l}.
\end{equation}
\end{lemma}

\begin{proof}
Assume $\esing_1$ and $\esing_2$ are coupled at iteration $\l$. Then, since $\esing_1$ and $\esing_2$ take the same steps on the interval $[a_{\l-1},a_\l]$, in order for them to uncouple at the next iteration, $\esing_{0,\l}$ must have hit~$\esing_1[0,a_{\l-1}]-\esing_1(a_\l)$ and not hit $\esing_2[0,a_{\l-1}]-\esing_1(a_\l)$ or the other way around, or $\esing_{0,\l}$ must have not escaped both $\esing_1[0,a_\l]$ and $\esing_2[0,a_{\l}]$. In the latter event, $\esing_{0,\l}$ is resampled, so the latter event is independent of the probability that the coupling is successful at iterations $\l$ and $\l+1$. Hence, by a union bound,
\begin{equation}\label{eq:unionbounduncoupling}
    \begin{split}
        &\P(\text{coupling is not successful at iteration $\l+1$},\, \text{coupling is successful at iteration $\l$})\\
        \leq&\P(\eta^\infty[1,\infty)\cap(\esing_1[0,a_{\l-1}]-\esing_1(a_\l))\neq\varnothing)\\
        &+\P(\eta^\infty[1,\infty)\cap(\esing_2[0,a_{\l-1}]-\esing_2(a_\l))\neq\varnothing)\\
        &+\P(\eta^\infty\nesc\esing_1[0,a_{\l}])\\
        &\cdot\P(\text{coupling is not successful at iteration $\l+1$},\, \text{coupling is successful at iteration $\l$})
    \end{split}
\end{equation}
We now bound each of the terms above. By Lemmas \ref{lemma:traveldistance} and \ref{lemma:twopoint} and a union bound, we can choose $a_{\l}$ sufficiently large depending on $a_{\l-1}$ such that
\begin{equation}
    \begin{split}
        &\P(\eta^\infty[1,\infty)\cap(\esing_1[0,a_{\l-1}]-\esing_1(a_\l))\neq\varnothing)\\
        \leq&\P(\eta^\infty[1,\infty)\cap(\esing_1[0,a_{\l-1}]-\esing_1(a_\l))\neq\varnothing,\,\|\esing_1(a_\l)\|\geq a_\l^{1/4})+\P(\|\esing_1(a_\l)\|<a_\l^{1/4})\\
        \lesssim&a_{\l-1}(a_\l^{1/4}-a_{\l-1})^{2-d}+e^{-\l}\\
        \lesssim&e^{-\l}.
    \end{split}
\end{equation}
The second term of \eqref{eq:unionbounduncoupling} can be bounded in the same way. Furthermore, by Lemma \ref{lemma:avoidingconditioned}, there exists $\varepsilon>0$ such that for all $\l$ large enough,
\begin{equation}
    \begin{split}
        \P(\eta^\infty\nesc\teta_1[0,a_{\l}])<1-\varepsilon.
    \end{split}
\end{equation}
Combining the above, we obtain
\begin{equation}
    \begin{split}
        \P(\text{coupling is not successful at iteration $\l+1$},\, \text{coupling is successful at iteration $\l$})\lesssim e^{-\l},
    \end{split}
\end{equation}
which is what we wanted to show.
\end{proof}

\subsubsection{Proof of ergodicity}\label{sec:proofoferg}
\begin{proposition}\label{prop:coupling}
Let $\esing_1,\esing_2$ be infinite SAWs conditioned to start with $\zeta_1$ and $\zeta_2$ respectively, coupled as above. Then there exist $L=L(\zeta_1,\zeta_2)$, $C=C(\zeta_1,\zeta_2)$ and a sequence $(a_\l)_{\l\geq0}$ such that for all $\l\geq L$
\begin{equation}
    \P(\text{coupling is not successful at iteration $\l$})\leq e^{-C\l}.
\end{equation}
\end{proposition}  

\begin{proof}
Let $a_\l$ be the maximum of $e^\l$ and the $a_\l$'s from Lemma \ref{lemma:probcouple}. The proof is by induction. There clearly exists an $L_1$ such that the probability that the walks are coupled at iteration $L_1$ is positive, so the base case is true. 

Assume the assertion is true for $\l$. Then by Lemmas \ref{lemma:probcouple} and \ref{lemma:probuncouple}, there exist $c,C_1$ such that
\begin{equation}
    \begin{split}
        &\P(\text{coupling is not successful at iteration $\l+1$})\\
        =&\P( \text{coupling is not successful at iteration $\l+1$},\,\text{coupling is successful at iteration $\l$})\\
        &+\P(\text{coupling is not successful at iteration $\l+1$}\mid \text{coupling is not successful at iteration $\l$})\\
        &\cdot\P(\text{coupling is not successful at iteration $\l$})\\
        \leq&C_1e^{-\l}+(1-c)e^{-C\l}\leq e^{-C(\l+1)}
    \end{split}
\end{equation}
if $C$ is chosen to be small enough.
\end{proof}

\begin{proof}[Proof of Proposition \ref{prop:ergodicity}]
Let $\esing_1$ and $\esing_2$ be independent infinite self-avoiding walks conditioned to start with paths $\zeta_1$ and $\zeta_2$ respectively, both of length $k$. Let $\eta_1'$ and $\eta_2'$ be the coupling of $\esing_1$ and $\esing_2$ as described in Section \ref{sec:constructioncoupling}. Then by Proposition \ref{prop:coupling},
\begin{equation}
	\begin{split}
		\limm\|T^m\esing_1-T^m\esing_2\|_{\TV}\leq&\limm\P(T^{m}\eta'_1\neq T^m\eta'_2)\\
		=&\lim_{L\rightarrow\infty}\P(T^{a_L}\eta'_1\neq T^{a_L}\eta'_2)\\
		\leq&\lim_{L\rightarrow\infty}\sum_{\l\geq L}\P(T^{a_{\l-1}}\eta'_1[0,a_{\l}-a_{\l-1}]\neq T^{a_{\l-1}}\eta'_2[0,a_{\l}-a_{\l-1}])=0,
	\end{split}
\end{equation}
which concludes the proof.
\end{proof}

\begin{remark}
Note that in the above proof, we in fact show that almost surely there exists $m$ such that $T^m\eta'_1=T^m\eta'_2$. This follows from Borel-Cantelli and the fact that $\sum_{\l\geq1}\P(T^{a_\l}\eta'_1\neq T^{a_\l}\eta'_2)<\infty$.
\end{remark}

\section{Convergence in probability}\label{sec:probabilityconv}
In this section, we prove \eqref{eq:conv-pattern-prob}.  To obtain convergence in probability, we show that the variance of $\pi_\zeta(n,\eta^n)$ tends to 0. In order to do this, we must decorrelate parts of the SAW that are far apart. The decorrelation strategy is exactly the same one as the proof of ergodicity for the infinite SAW. We define a finite, two-sided version of the coupling from Section \ref{sec:coupling}. As before, the coupling has probability bounded away from 0 to be successful at each iteration and once the walks are successfully coupled, the probability that future iterations are unsuccessful is small. Since the proofs are similar, we omit many details.

In Section \ref{sec:constructioncouplingtwosided}, we construct the two-sided coupling. We show asymptotic independence of far away parts of the SAW in Section \ref{sec:decorrelationtwosided}. Finally, we give the proof of \eqref{eq:conv-pattern-prob} in Section \ref{sec:convergenceinprob}.

\subsection{Two-sided coupling}\label{sec:constructioncouplingtwosided}
We define a coupling for finite, two-sided SAWs. Let $\eta^{m,n}_1$ and $\eta^{m,n}_2$ be finite, two-sided SAWs conditioned such that $\eta_1^{m,n}[-k,k]=\zeta_1$ and $\eta^{m,n}_2[-k,k]=\zeta_2$. Let $(a_\l)_\l$ be an increasing sequence with $a_0=k$. Assume that we have defined $\eta_{1}^{m,n}[-a_\l,a_\l]$ and $\eta_{2}^{m,n}[-a_\l,a_\l]$. To obtain the next iteration, we sample two independent one-sided SAWs $\eta_{-,\l}:=\eta_-^{(m-a_{\l})_+}$ and $\eta_{+,\l}:=\eta_+^{(n-a_\l)_+}$.

If appending $\eta_{-,\l}$ and $\eta_{+,\l}$ to respectively the start and end of $\eta_1^{m,n}[-a_\l,a_\l]$ results in a self-avoiding path of length $m+n$ which we call $\eta_0^{m,n}$, then set $\eta_1^{m,n}[-a_{\l+1},a_{\l+1}]=\eta_0^{m,n}[-a_{\l+1},a_{\l+1}]$. Do the same for $\eta_2^{m,n}$. If both paths are not self-avoiding, we resample $\eta_{-,\l}$ and $\eta_{+,\l}$. If one of the two paths is not self-avoiding, we resample that walk independently.

So $\eta^{m,n}_1$ and $\eta^{m,n}_2$ are successfully coupled when appending $\eta_{-,\l}$ and $\eta_{+,\l}$ to the starts and ends of $\eta_2^{m,n}[-a_\l,a_\l]$ and $\eta_2^{m,n}[-a_\l,a_\l]$ results in two self-avoiding paths. The coupling is unsuccessful if one of the paths is self-avoiding and the other one is not. If both are not self-avoiding, we resample. If $a_\l\geq m$ or $a_\l\geq n$, we simply let $\eta_{-,\l}$ or $\eta_{+,\l}$ respectively be a path of length 0. This procedure is repeated until $a_\l\geq m\vee n$.

Note that a two-sided SAW $\eta^{m,n}$ conditional on $\eta^{m,n}[-k,k]=\zeta$ is distributed as the concatenation of a SAW of length $m-k$ started from~$\zeta(-k)$, $\zeta$, and a SAW of length $n-k$ started from~$\zeta(k)$, all conditioned not to intersect each other. It follows that, under this coupling, $\eta_1^{m,n}$ and $\eta_2^{m,n}$ are indeed distributed as finite two-sided SAWs conditioned to start with $\zeta_1$ and $\zeta_2$ respectively. Convergence in probability of~$\pi_{\zeta}(n,\eta^n)$ will follow from the following proposition, which we prove in the next section.

\begin{proposition}\label{prop:ergodicitytwosided}
For all $\zeta\in\SAW_k$ and $\varepsilon>0$, there exists $i_0=i_0(k,\varepsilon)$  and $m_0=m_0(k,\varepsilon)$ such that for all $i\geq i_0$ and $m,n\geq m_0$,
\begin{equation}
    \P(\eta_1^{m,n}[-i,-i-k]\neq\eta_2^{m,n}[-i,-i-k]\text{ or }\eta_1^{m,n}[i,i+k]\neq\eta_2^{m,n}[i,i+k])<\varepsilon.
\end{equation}
\end{proposition}
This proposition says that far away from 0, $\eta_1^{m,n}$ and $\eta_2^{m,n}$ will be successfully coupled with high probability, uniformly in $m$ and $n$. So far away parts of the two-sided SAW have small correlations. Since the finite two-sided walk is simply a shift of the finite one-sided walk, it follows that far away parts of the finite SAW have small correlations. The law of large numbers will follow from this observation.

\subsection{Decorrelation}\label{sec:decorrelationtwosided}
We first prove a two-sided analogue of Lemma \ref{lemma:avoidingconditioned}. The proof is very similar. It is important to note that the bound is uniform in $m$ and $n$.
\begin{lemma}\label{lemma:avoidingconditionedtwosided}
Let $\eta^{m,n}$ be a two-sided self-avoiding walk and $\eta_1^{m-l}$, $\eta_2^{n-l}$ be independent, one-sided self-avoiding walks. For all $\varepsilon>0$ and $k$, there exists $L=L(\varepsilon,k)$ such that for all two-sided self-avoiding paths $\xi$ of length $2k$, $l\geq L$ and $m,n\geq l+\sqrt{l}$,
\begin{equation}
    \P(\eta_1^{m-l}[0,m-l]\oplus\eta^{m,n}[-l,l]\oplus\eta_2^{n-l}[0,n-l]\issaw\mid\eta^{m,n}[-k,k]=\xi)>(1-\varepsilon)A^{-2}.
\end{equation}
\end{lemma}
\begin{proof}
We write $\zeta\succeq\xi$ if $\zeta$ is a two-sided path such that $\zeta[-k,k]=\xi$ and we write $c_{m,n}(\xi)$ for the number of two-sided paths $\zeta$ of length $m+n$ such that $\zeta\succeq\xi$. Then
\begin{equation}
    \begin{split}
        &\P(\eta_1^{m-l}[0,m-l]\oplus\eta^{m,n}[-l,l]\oplus\eta_2^{n-l}[0,n-l]\issaw\mid\eta^{m,n}[-k,k]=\xi)\\
        =&\sum_{\zeta\succeq\xi}\frac{c_{m,n}(\zeta)}{c_{m-l}c_{n-l}}\P(\eta^{m,n}[-l,l]=\zeta\mid\eta^{m,n}[-k,k]=\xi)\\
        =&\sum_{\zeta\succeq\xi}\frac{c_{m+n}\P(\eta^{m,n}[-l,l]=\zeta)}{c_{m-l}c_{n-l}}\cdot\frac{\P(\eta^{m,n}[-l,l]=\zeta)}{\P(\eta^{m,n}[-k,k]=\xi)}\\
        \geq&\frac{c_{m+n}}{c_{m-l}c_{n-l}\P(\eta^{m,n}[-k,k]=\xi)c_{l,l}(\xi)}\left(\sum_{\zeta\succeq\xi}\P(\eta^{m,n}[-l,l]=\zeta)\right)^2\\
        =&\frac{c_{m+n}\P(\eta^{m,n}[-k,k]=\xi)}{c_{m-l}c_{n-l}c_{2l}\P(\eta^{l,l}[-k,k]=\xi)}.
    \end{split}
\end{equation}
The inequality follows from Jensen. Note that since $m,n\geq l+\sqrt{l}$, we have that $m-l\rightarrow\infty$ and $n-l\rightarrow\infty$ as $l\rightarrow\infty$, so $\frac{c_{m+n}}{c_{m-l}c_{n-l}c_{2l}}\rightarrow A^{-2}$ as $l\rightarrow\infty$ by Theorem \ref{th:countingwalks}. Furthermore, $\frac{\P(\eta^{m,n}[-k,k]=\xi)}{\P(\eta^{l,l}[-k,k]=\xi)}\rightarrow1$ as $l\rightarrow\infty$ by Theorem \ref{th:maintwosided}, which completes the proof.
\end{proof}

We then prove a two-sided analogue of Lemma \ref{lemma:probcouple}. The proof is almost identical, so we only point out what needs to be altered.
\begin{lemma}\label{lemma:probcoupletwosided}
Let $\eta_1^{m,n}$ and  $\eta_2^{m,n}$ be two-sided SAWs conditioned such that $\eta_1^{m,n}[-k,k]=\zeta_1$ and $\eta_2^{m,n}[-k,k]=\zeta_2$, coupled as above. Then there exists $c>0$ and a sequence $(a_\l)_{\l\geq0}$ such that for all $m,n$ and all $\l\geq 2$ with $a_{\l}+\sqrt{a_{\l}}\leq m\wedge n$,
\begin{equation}
    \P(\text{coupling is successful at iteration $\l+1$}\mid\text{coupling is unsuccessful at iteration $\l$})>c.
\end{equation}
\end{lemma}

\begin{proof}
If $\eta^{m,n}_1$ and $\eta^{m,n}_2$ are not successfully coupled at iteration $\l$, then it means that $\eta_{+,\l}$ at iteration $\l$ escaped $\eta_{-,\l}\oplus\eta_2^{m,n}[-a_{\l-1},a_{\l-1}]$ but not $\eta_{-,\l}\oplus\eta_1^{m,n}[-a_{\l-1},a_{\l-1}]$ or the other way around. 
Let $\tilde{\eta}^{m,n}_1$ and $\tilde{\eta}^{m,n}_2$ be two independent two-sided SAWs. Using the hitting estimates from Section \ref{sec:hitting}, it can be shown that there exists $\Tilde{a}_\l$ depending only on $a_{\l-1}$ such that
\begin{equation}
    \begin{split}
        &\min_{\xi_1,\xi_2}\P(\tilde{\eta}^{m,n}_1[a_{\l-1},\tilde{a}_\l]\nesc\tilde{\eta}^{m,n}_2[-\tilde{a}_\l,a_{\l-1}]\\
        &\mid\tilde{\eta}_1[-a_{\l-1},a_{\l-1}]=\xi_1,\,\tilde{\eta}_2[-a_{\l-1},a_{\l-1}]=\xi_2,\,\tilde{\eta}^{m,n}_1[a_{\l-1},n]\nesc\tilde{\eta}^{m,n}_2[-m,a_{\l-1}])\geq\frac{1}{2},
    \end{split}
\end{equation}
where the minimum is taken over all paths $\xi_1,\xi_2$ of length $2a_{\l-1}$ such that the conditioning event has positive probability. The rest of the proof then proceeds as for Lemma \ref{lemma:probcouple}, choosing $a_{\l}$ larger than $L(\varepsilon,\tilde{a}_\l)$ from Lemma \ref{lemma:avoidingconditionedtwosided} for $\varepsilon$ small enough.
\end{proof}

The following lemma is the analogue of Lemma \ref{lemma:probuncouple}. Like before, the proof follows from the hitting estimates in Section \ref{sec:hitting}. We leave the details to the reader.
\begin{lemma}\label{lemma:probuncoupletwosided}
Let $\eta_1^{m,n}$ and  $\eta_2^{m,n}$ be two-sided SAWs conditioned such that $\eta_1^{m,n}[-k,k]=\zeta_1$ and $\eta_2^{m,n}[-k,k]=\zeta_2$, coupled as above. Then there exists $C=C(\zeta_1,\zeta_2)$ and a sequence $(a_\l)_{\l\geq1}$ such that for all $m,n$ and all $\l\geq2$ with $a_\l\leq m\wedge n$,
\begin{equation}
    \P(
    \text{coupling is unsuccessful at iteration $\l+1$},\text{coupling is successful at iteration $\l$})\leq e^{-C\l}.
\end{equation}
\end{lemma}

\begin{proof}
Assume $\eta_1^{m,n}$ and $\eta_2^{m,n}$ are successfully coupled at iteration $\l$. Then, since $\eta_1^{m,n}$ and $\eta_2^{m,n}$ take the same steps on the intervals $[-a_{\l-1},-a_\l]$ and $[a_{\l-1},a_\l]$, in order for them to uncouple at the next iteration, $\eta_{+}^{n-a_\l}$ must have hit $\eta_{-}^{m-a_\l}\oplus\eta_1^{m,n}[-a_{\l-1},a_{\l-1}]$ and not hit $\eta_{-}^{m-a_\l}\oplus\eta^{m,n}_2[-a_{\l-1},a_{\l-1}]$ or the other way around, or $\esing_{0,\l}$ must have not escaped both $\eta_1[0,a_\l]$ and $\eta_2[0,a_{\l}]$ and was resampled. We can make the probability of the former events arbitrary small by choosing $a_\l$ large enough. In particular, the probability can be made smaller than $e^{-C\l}$.
\end{proof}

We conclude that the probability that the coupling is unsuccessful decreases exponentially.
\begin{proposition}\label{prop:couplingtwosided}
Let $\eta_1^{m,n}$ and  $\eta_2^{m,n}$ be two-sided SAWs conditioned such that $\eta_1^{m,n}[-k,k]=\zeta_1$ and $\eta_2^{m,n}[-k,k]=\zeta_2$, coupled as above. Then there exist $L=L(\zeta_1,\zeta_2)$, $C=C(\zeta_1,\zeta_2)$ and a sequence $(a_\l)_{\l\geq0}$ such that for all $m,n$ and all $\l\geq L$ with $a_\l+\sqrt{a_\l}\leq m\wedge n$,
\begin{equation}
    \P(\text{coupling is unsuccessful at iteration $\l$})\leq e^{-C\l}.
\end{equation}
\end{proposition}  

\begin{proof}
The proof is exactly the same as the proof of Proposition \ref{prop:coupling}, choosing the sequence $a_\l$ to be the maximum of the sequences from Lemmas \ref{lemma:probuncoupletwosided} and \ref{lemma:probcoupletwosided}.
\end{proof}

\begin{proof}[Proof of Proposition \ref{prop:ergodicitytwosided}]
Let $\eta_1^{m,n}$ and $\eta_2^{m,n}$ be two-sided SAWs conditioned such that~$\eta_1^{m,n}[-k,k]=\zeta_1$ and $\eta_2^{m,n}[-k,k]=\zeta_2$, coupled as above. Let $\l$ such that $i\in[a_{\l-1},a_{\l}]$. Then for all $i$ such that~$e^{-C\l}<\varepsilon$ and all $m,n\geq a_\l+\sqrt{a_\l}$, we have
\begin{equation}
    \begin{split}
        &\P(\eta_1^{m,n}[-i,-i-k]\neq\eta_2^{m,n}[-i,-i-k]\text{ or }\eta_1^{m,n}[i,i+k]\neq\eta_2^{m,n}[i,i+k])\\
        &\leq\P(\text{coupling is unsuccessful at iteration $\l$})\leq e^{-C\l}<\varepsilon,
    \end{split}
\end{equation}
which concludes the proof.
\end{proof}

\subsection{Convergence in probability}\label{sec:convergenceinprob}
\begin{proof}[Proof of \Cref{eq:conv-pattern-prob} of Theorem \ref{th:mainpatterns}]
Let $\eta_1^n,\eta_2^n$ be two independent SAWs. Let $\varepsilon>0$. Let $i_0=i_0(k,\varepsilon)$ as in Proposition \ref{prop:ergodicitytwosided}. Let $i_0\leq i,j\leq n$ such that $|i-j|\geq i_0$ and $j,n-j\geq m_0$. Let $\eta_1^{j,n-j},\eta_2^{j,n-j}$ be two-sided SAWs conditioned to start with $\xi_1,\xi_2$, coupled as in Section \ref{sec:constructioncouplingtwosided}. Then by Proposition~\ref{prop:ergodicitytwosided}, \begin{equation}
    \begin{split}
        &\left|\E\left[\1_{\{\eta_1^n[i,i+k]=\zeta\}}\mid\eta_1^n[j,j+k]=\xi_1\right]-\E\left[\1_{\{\eta_2^n[i,i+k]=\zeta\}}\mid\eta_2^n[j,j+k]=\xi_2\right]\right|\\
        =&\left|\E\left[(\1_{\{\eta_1^{j,n-j}[i-j,i-j+k]=\zeta\}}-\1_{\{\eta_2^{j,n-j}[i-j,i-j+k]=\zeta\}})\mid\eta_1^{j,n-j}[0,k]=\xi_1,\,\eta_2^{j,n-j}[0,k]=\xi_2\right]\right|<\varepsilon.
    \end{split}
\end{equation}
It follows that
\begin{equation}
    \begin{split}
        &\Cov(\1_{\{\eta^n[i,i+k]=\zeta\}},\1_{\{\eta^n[j,j+k]=\zeta\}})<\varepsilon.
    \end{split}
\end{equation}
Hence,
\begin{equation}
    \begin{split}
        \Var(\pi_\zeta(\eta^n,n))=&\frac{1}{n^2}\sum_{i,j=1}^n\Cov(\1_{\{\eta^n[i,i+k]=\zeta\}},\1_{\{\eta^n[j,j+k]=\zeta\}})\\
        \leq&\frac{1}{n^2}\sum_{i=i_0}^{n}\sum_{\substack{j=i_0\vee m_0\\|j-i|>i_0}}^{n-m_0}\Cov(\1_{\{\eta^n[i,i+k]=\zeta\}},\1_{\{\eta^n[j,j+k]=\zeta\}})+o(1)\\
        \leq&\varepsilon+o(1),\qquad n\rightarrow\infty.
    \end{split}
\end{equation}
Since $\varepsilon>0$ was chosen to be arbitrary, we conclude that 
\begin{equation}
    \limn\Var(\pi_\zeta(\eta^n,n))=0,
\end{equation}
which implies that $\pi_\zeta(\eta^n,n)$ converges in probability. 
\end{proof}

\section{Infinite SAW and the domain Markov property}\label{sec:conj}
In lower dimensions, the existence of one-sided infinite SAW has not been proved. Since proving that the weak limit of finite SAW exists is a difficult task, we propose an alternative definition of the one-sided infinite SAW in this section that does not directly involve taking a limit. In the following, we let $\eta$ denote any random variable taking values in $\SAW_n$ for some $n$ (including $n=\infty$) and let $d\geq1$.

\begin{conjecture}\label{con:dmp}
For all $d\geq1$, there exists a unique probability measure $P$ on $\SAW_\infty$ such that $\P(\eta(1)=e)=\frac{1}{2d}$ for all neighbours $e$ of the origin and that satisfies the domain Markov property, i.e, for all events $E$,
\begin{equation}
    \begin{split}
        P(\eta[k,\infty)\in E\mid\eta[0,k]=\zeta')=P(\eta[0,\infty)\in E\mid\eta[1,\infty)\cap\zeta'=\varnothing,\,\eta(0)=\zeta(k)).
    \end{split}
\end{equation}
\end{conjecture}
We say a measure $P$ is \emph{symmetric} if $P(\eta(1)=e)=\frac{1}{2d}$ for all neighbours $e$. We further say a probability measure satisfies the \emph{symmetric domain Markov property} if it is symmetric and satisfies the DMP. Therefore, if Conjecture \ref{con:dmp} holds, we can simply define the one-sided infinite self-avoiding walk to be the unique random variable satisfying the symmetric domain Markov property. To prove that this random variable is indeed the weak limit of finite SAWs, it remains to show that every subsequential limit of finite SAWs satisfies the symmetric DMP. Of course, it is not clear that this approach is easier than a more direct proof that finite SAW converges weakly in low dimensions. However, even if it is not possible to prove that subsequential limits satisfy the symmetric DMP, a proof of Conjecture \ref{con:dmp} would still at least give a candidate weak limit of finite SAW. This is similar in spirit to the conjecture that $\text{SLE}_{\frac{8}{3}}$ is the scaling limit of SAW in two dimensions \cite{lawler04scaling}.

For $d\geq5$, this approach does in fact lead to a novel proof of existence of infinite SAW. We give this proof in Section \ref{sec:conhighdim}, see Theorem \ref{th:conhighdim}. In contrast to Lawler's proof from \cite{lawler89infinite}, this proof does not directly use the lace expansion, but only relies on the DMP and on the hitting estimates from Section \ref{sec:hitting}.

In Section \ref{sec:equivdmp}, we give an equivalent characterization of the symmetric DMP and provide an alternative version of Conjecture~\ref{con:dmp}. In Section \ref{sec:conhighdim}, we prove the alternative version of Conjecture \ref{con:dmp} for $d\geq5$. In doing so, we give a new proof of existence of infinite SAW.

\subsection{An alternative domain Markov property}\label{sec:equivdmp}
We can show the following equivalent characterizations of the symmetric DMP.
\begin{proposition}\label{prop:dmpequiv}
Let $P$ be a symmetric measure on $\SAW_\infty$. Let $e$ be any neighbour of the origin and let $\mu(P)=2dP(0\not\in\eta[0,\infty)\mid\eta(0)=e)$. The following are equivalent:
\begin{enumerate}
	\item\label{prop:dmp} $P$ satisfies the DMP.\\
	\item\label{prop:altdmp} Let $\zeta$ be a self-avoiding path of length $k$ and $x\sim\eta(k)$. Then
	\begin{equation}
		\begin{split}
			P(\eta(k+1)=x\mid\eta[0,k]=\zeta)=&\frac{P(\eta[1,\infty)\cap\zeta=\varnothing\mid\eta(0)=\zeta(k),\eta(1)=x)}{\sum_{y\sim\zeta(k)}P(\eta[1,\infty)\cap\zeta=\varnothing\mid\eta(0)=\zeta(k),\eta(1)=y)}.
		\end{split}
	\end{equation}
	\item\label{prop:altdmp2} For all self-avoiding paths $\zeta\in\SAW_k$ and events $E$,
	\begin{equation}
		\begin{split}
			P(\eta[0,k]=\zeta,\,(T^k\eta)[0,\infty)\in E)=\mu(P)^{-k}P(\eta\esc\zeta,\,\eta[0,\infty)\in E).
		\end{split}
	\end{equation}
\end{enumerate}
\end{proposition}
\begin{proof}
Assume $P$ satisfies the symmetric DMP. We prove \ref{prop:altdmp} is true. By the DMP,
\begin{equation}
	\begin{split}
		&P(\eta(k+1)=x\mid\eta[0,k]=\zeta)=P(\eta(1)=x\mid\eta(0)=\zeta(k),\,\eta[1,\infty)\cap\zeta=\varnothing)\\
		=&\frac{P(\eta[1,\infty)\cap\zeta=\varnothing\mid\eta(0)=\zeta(k),\,\eta(1)=x)P(\eta(1)=x\mid\eta(0)=\zeta(k))}{P(\eta[1,\infty)\cap\zeta=\varnothing\mid\eta(0)=\zeta(k))}\\
		=&\frac{P(\eta[1,\infty)\cap\zeta=\varnothing\mid\eta(0)=\zeta(k),\,\eta(1)=x)}{2d\sum_{y\sim\zeta(k)}P(\eta[1,\infty)\cap\zeta=\varnothing\mid\eta(0)=\zeta(k),\eta(1)=y)P(\eta(1)=y\mid\eta(0)=\zeta(k))}\\
		=&\frac{P(\eta[1,\infty)\cap\zeta=\varnothing\mid\eta(0)=\zeta(k),\,\eta(1)=x)}{\sum_{y\sim\zeta(k)}P(\eta[1,\infty)\cap\zeta=\varnothing\mid\eta(0)=\zeta(k),\eta(1)=y)},
	\end{split}
\end{equation}
which is what we wanted to show. By the same calculation, if \ref{prop:altdmp} is true, then 
\begin{equation}
	\begin{split}
		P(\eta(k+1)=x\mid\eta[0,k]=\zeta)=&\frac{P(\eta[1,\infty)\cap\zeta=\varnothing\mid\eta(0)=\zeta(k),\,\eta(1)=x)}{\sum_{y\sim\zeta(k)}P(\eta[1,\infty)\cap\zeta=\varnothing\mid\eta(0)=\zeta(k),\eta(1)=y)}\\
		=&P(\eta(1)=x\mid\eta(0)=\zeta(k),\,\eta[1,\infty)\cap\zeta=\varnothing),
	\end{split}
\end{equation}
which implies that $P$ satisfies the DMP.

So \ref{prop:dmp} and \ref{prop:altdmp} are equivalent. Assume that \ref{prop:dmp} and \ref{prop:altdmp} are true. Then by the DMP,
\begin{equation}
	\begin{split}
		P(\eta(k+1)=x\mid\eta[0,k]=\zeta)=&\frac{P(\eta[1,\infty)\cap\zeta=\varnothing\mid\eta(0)=\zeta(k),\,\eta(1)=x)}{\sum_{y\sim\zeta(k)}P(\eta[1,\infty)\cap\zeta=\varnothing\mid\eta(0)=\zeta(k),\eta(1)=y)}\\
		=&\frac{P(\eta[0,\infty)\cap\zeta\mid\eta(0)=x,\,\eta[1,\infty)\cap\{\zeta(k),x\}=\varnothing)}{2dP(\eta\esc\zeta)}\\
		=&\frac{P(\eta\esc\zeta\cup\{x\})}{2dP(\zeta(k)\not\in\eta[1,\infty)\mid\eta(0)=x)P(\eta\esc\zeta)}\\
		=&\frac{P(\eta\esc\zeta\cup\{x\})}{\mu(P)P(\eta\esc\zeta)}.
	\end{split}
\end{equation}
Thus, for all self-avoiding paths $\xi$,
\begin{equation}
	\begin{split}
		&P(\eta[0,k]=\zeta,(T^k\eta)[0,m]=\xi)\\
		=&\prod_{i=1}^kP(\eta(i)=\zeta(i)\mid\eta[0,i-1]=\zeta[0,i-1])P((T^k\eta)[0,m]=\xi\mid\eta[0,k]=\zeta)\\
		=&\prod_{i=1}^k\mu(P)^{-1}\frac{P(\eta\esc\zeta[0,i])}{P(\eta\esc\zeta[0,i-1])}P(\eta[0,m]=\xi\mid\eta\esc\zeta)\\
		=&\mu(P)^{-k}P(\eta\esc\zeta)P(\eta[0,m]=\xi\mid\eta\esc\zeta)\\
		=&\mu(P)^{-k}P(\eta[0,m]=\xi,\,\eta\esc\zeta),
	\end{split}
\end{equation}
which implies \ref{prop:altdmp2}.

Finally, we saw in the proof of Theorem \ref{th:dmp} that \ref{prop:altdmp2} implies \ref{prop:dmp}, which concludes the proof.
\end{proof}

In light of the above, we define the operator $F$ on the space of probability measures on $\SAW_\infty$ given by
\begin{equation}
	\begin{split}
		(FP)(\eta(k+1)=x\mid\eta[0,k]=\zeta)=\frac{P(\eta[1,\infty)\cap\zeta=\varnothing\mid\eta(0)=\zeta(k),\eta(1)=x)}{\sum_{y\sim\zeta(k)}P(\eta[1,\infty)\cap\zeta=\varnothing\mid\eta(0)=\zeta(k),\eta(1)=x)}.
	\end{split}
\end{equation}
Then by Proposition \ref{prop:dmpequiv}.(ii), Conjecture \ref{con:dmp} is equivalent to the statement that $F$ has a unique fixed point in the space of symmetric probability measures on $\SAW_\infty$.

Furthermore, we can slightly alter Conjecture \ref{con:dmp} to the statement that there exists a unique probability measure $P$ satisfying the symmetric DMP and such that $\mu(P)=\mu$, with $\mu$ the connective constant.
\begin{conjecture}\label{con:infalt}
There exists a unique probability measure $P$ on infinite paths satisfying for all $k$ and all $\zeta\in\SAW_k$,
\begin{equation}
	P(\eta[0,k]=\zeta)=\mu^{-k}P(\eta\esc\zeta).
\end{equation}
\end{conjecture}
Note that the above is not strictly stronger or weaker than Conjecture \ref{con:dmp}. However, Conjecture~\ref{con:infalt} should still uniquely characterize the infinite SAW in the sense that if both the conjectures are true, then the unique measure from Conjecture \ref{con:dmp} is the same as the unique measure from \ref{con:infalt}. Furthermore, this unique measure should then also be the weak limit of finite SAW. In the next section, we prove Conjecture~\ref{con:infalt} in dimensions $d\geq5$.

\subsection{Proof of Conjecture \ref{con:infalt} in high dimensions}\label{sec:conhighdim}
In this section, we show that Conjecture \ref{con:infalt} is true for $d\geq5$. We also show that the unique measure satisfying $\eqref{eq:dmpaltsing}$ is the weak limit of finite SAW. For the proof, we do not assume that infinite SAW exists. Thus, we give a new proof of existence of infinite SAW. The proof relies solely on the asymptotics for $c_n$ and the two-point function $G(x)$.

\begin{theorem}\label{th:conhighdim}
For $d\geq5$, Conjecture \ref{con:infalt} is true and the unique measure satisfying \eqref{eq:dmpaltsing} is the law of the one-sided infinite SAW.
\end{theorem}

We first prove that finite SAW converges weakly. Note that the space $\SAW_\infty$ is compact in the weak topology, so the space of probability measures on $\SAW_\infty$ in the Borel sigma-algebra generated by the weak topology is also compact in the topology of weak convergence. So every subsequence of $(\eta^n)_{n\in\N}$ has a weakly convergent further subsequence. It suffices to show that all these subsequences converge to the same limit.

Let $(\alpha_n)_{n\in\N}$ and $(\beta_n)_{n\in\N}$ be increasing sequences in $\N$ and let $P_1,P_2$ be probability measures on $\SAW_\infty$ such that the laws of $(\eta^{\alpha_n})_{n\in\N}$ and $(\eta^{\beta_n})_{n\in\N}$ converge weakly to $P_1$ and $P_2$ respectively. The goal is to show that $P_1=P_2$.

Recall that $\eta^n$ refers to the $n$-step SAW. As usual, we let $\P$ denote the law of $\eta^n$. In the same way as in the proof of Theorem \ref{th:dmpalt}, we have that
\begin{equation}
	\begin{split}
		P_1(\eta[0,k]=\zeta)=&\limm\P(\eta^{\alpha_m}[0,k]=\zeta)\\
		=&\limm\frac{c_{\alpha_m-k}}{c_{\alpha_m}}\P(\eta^{\alpha_m-k}\esc\zeta)\\
		=&\mu^{-k}\lim_{m\rightarrow\infty}\P(\eta^{\alpha_m-k}\esc\zeta).
	\end{split}
\end{equation}
Note that we cannot conclude from the convergence of $(\eta^{\alpha_m})_m$ that $(\eta^{\alpha_m-k})_{m}$ converges, so we cannot assume that the last probability equals $P_1(\eta\esc\zeta)$. However, the above formula will still be useful.

We denote $\xi\succeq\zeta$ if $\xi$ is a $\beta_n$-step self-avoiding path starting with $\zeta$. Then, for all $\beta_n$,
\begin{equation}
	\begin{split}
		P_1(\eta[0,k]=\zeta)=&\sum_{\xi\succeq\zeta}P_1(\eta[0,\beta_n]=\xi)\\
		=&\sum_{\xi\succeq\zeta}\mu^{-\beta_n}\lim_{m\rightarrow\infty}\P(\eta^{\alpha_m-\beta_n}\esc\xi)\\
		=&\frac{c_{\beta_n}(\zeta)}{\mu^{\beta_n}}\sum_{\xi\succeq\zeta}\frac{1}{c_{\beta_n}\P(\eta^{\beta_n}[0,k]=\zeta)}\lim_{m\rightarrow\infty}\P(\eta^{\alpha_m-\beta_n}\esc\xi)\\
		=&\frac{c_{\beta_n}(\zeta)}{\mu^{\beta_n}}\sum_{\xi\succeq\zeta}\P(\eta^{\beta_n}[0,\beta_n]=\xi\mid\eta^{\beta_n}[0,k]=\zeta)\lim_{m\rightarrow\infty}\P(\eta^{\alpha_m-\beta_n}\esc\xi)\\
		=&\frac{c_{\beta_n}(\zeta)}{\mu^{\beta_n}}\lim_{m\rightarrow\infty}\P(\eta^{\alpha_m-\beta_n}\esc\eta^{\beta_n}\mid\eta^{\beta_n}[0,k]=\zeta)\\
		=&\P(\eta^{\beta_n}[0,k]=\zeta)\frac{c_{\beta_n}}{\mu^{\beta_n}}\lim_{m\rightarrow\infty}\P(\eta^{\alpha_m-\beta_n}\esc\eta^{\beta_n}\mid\eta^{\beta_n}[0,k]=\zeta).
	\end{split}
\end{equation}
Note that $\limm\P(\eta^{\alpha_m-\beta_n}\esc\eta^{\beta_n})=\frac{\mu^{\beta_n}}{c_{\beta_n}}$ and $\P(\eta^{\beta_n}[0,k]=\zeta)$ tends to $P_2(\eta[0,k]=\zeta)$ as $n\rightarrow\infty$. So the goal is to show that 
\begin{equation}\label{eq:forgetbeginning}
	\lim_{m\rightarrow\infty}\P(\eta^{\alpha_m-\beta_n}\esc\eta^{\beta_n}\mid\eta^{\beta_n}[0,k]=\zeta)\sim \limm\P(\eta^{\alpha_m-\beta_n}\esc\eta^{\beta_n}),\qquad n\rightarrow\infty.
\end{equation}
So in a sense, we need to show that $\eta^{\beta_n}$ `forgets' its initial conditioning, which is reminiscent of our earlier decorrelation results. Indeed, this expression will follow using the exact same coupling technique as before, but without assuming existence of the infinite SAW. We give details on how to change the arguments from Sections \ref{sec:coupling} and \ref{sec:probabilityconv}. Some parts of the proof are entirely analogous to the earlier proofs, and we leave those details to the reader.

We have for all $0\leq t\leq \beta_n$,
\begin{equation}
	\begin{split}
		&\limm\P(\eta^{\alpha_m-\beta_n}\esc\eta^{\beta_n}[0,\beta_n]\mid\eta^{\beta_n}[0,k]=\zeta)\\
		=&1-\limm\P(\eta^{\alpha_m-\beta_n}[1,\alpha_m-\beta_n]\cap\eta^{\beta_n}[t,\beta_n]\neq\varnothing\mid\eta^{\alpha_m-\beta_n}(0)=\eta^{\beta_n}(\beta_n),\,\eta^{\beta_n}[0,k]=\zeta)\\
		&-\limm\P(\eta^{\alpha_m-\beta_n}[1,\alpha_m-\beta_n]\cap\eta^{\beta_n}[0,t)\neq\varnothing,\,\eta^{\alpha_m-\beta_n}[1,\alpha_m-\beta_n]\cap\eta^{\beta_n}[t,\beta_n]=\varnothing\\
		&\mid\eta^{\alpha_m-\beta_n}(0)=\eta^{\beta_n}(\beta_n),\eta^{\beta_n}[0,k]=\zeta).
	\end{split}
\end{equation}
We first show that the last term tends to 0 as $\beta_n\rightarrow\infty$. The hitting estimates from Section \ref{sec:hitting} are uniform in the length of the walk, which means that they also hold for any subsequential limit $\alpha_m-\beta_n\rightarrow\infty$. Hence, for fixed $t$, by Lemmas \ref{lemma:traveldistance} and \ref{lemma:twopoint},
\begin{equation}
	\begin{split}
		&\limm\P(\eta^{\alpha_m-\beta_n}[1,\alpha_m-\beta_n]\cap\eta^{\beta_n}[0,t)\neq\varnothing,\,\eta[1,\alpha_m-\beta_n]\cap\eta^{\beta_n}[t,\beta_n]=\varnothing\\
		&\mid\eta^{\alpha_m-\beta_n}(0)=\eta^{\beta_n}(\beta_n),\eta^{\beta_n}[0,k]=\zeta)\\
		\leq&\limm\P(\eta^{\alpha_m-\beta_n}[1,\alpha_m-\beta_n]\cap\eta^{\beta_n}[0,t)\neq\varnothing\mid\eta^{\alpha_m-\beta_n}(0)=\eta^{\beta_n}(\beta_n),\eta^{\beta_n}[0,k]=\zeta)\\
		\leq&\limm\P(\eta^{\alpha_m-\beta_n}[1,\alpha_m-\beta_n]\cap\eta^{\beta_n}[0,t]\neq\varnothing,\,\|\eta^{\beta_n}(\beta_n)\|>\beta_n^{1/4}\mid\eta(0)=\eta^{\beta_n}(\beta_n),\eta^{\beta_n}[0,k]=\zeta)\\
		&+\P(\|\eta^{\beta_n}(\beta_n)\|\leq\beta_n^{1/4}\mid\eta^{\beta_n}[0,k]=\zeta)\rightarrow0
	\end{split}
\end{equation}
as $\beta_n\rightarrow\infty$.

It remains to show that for all $\varepsilon>0$, there exists $t_0=t_0(\zeta,\varepsilon)$ such that for all $t\geq t_0$ and $\beta_n\geq 2t$,
\begin{equation}\label{eq:escapedecorrelate}
	\begin{split}
		&\limm\P(\eta^{\alpha_m-\beta_n}[1,\alpha_m-\beta_n]\cap\eta^{\beta_n}[t,\beta_n]\neq\varnothing\mid\eta^{\alpha_m-\beta_n}(0)=\eta^{\beta_n}(\beta_n),\,\eta^{\beta_n}[0,k]=\zeta)\\
		=&\limm\P(\eta^{\alpha_m-\beta_n}[1,\alpha_m-\beta_n]\cap\eta^{\beta_n}[t,\beta_n]\neq\varnothing\mid\eta^{\alpha_m-\beta_n}(0)=\eta^{\beta_n}(\beta_n))\pm\varepsilon.
	\end{split}
\end{equation}
Then \eqref{eq:forgetbeginning} follows, which completes the proof.

The required decorrelation is a direct consequence of the following proposition.
\begin{proposition}\label{prop:couplingfiniteonesided}
Let $\zeta_1,\zeta_2\in\SAW_k$ and let $\eta^{\beta_n}_1,\eta_2^{\beta_n}$ be two $\beta_n$-step SAWs conditioned to start with $\zeta_1$ and $\zeta_2$ respectively. Then there exists a coupling of $\eta^{\beta_n}_1$ and $\eta^{\beta_n}_2$ such that 
\begin{equation}
\lim_{t\rightarrow\infty}\sup_{\beta_n\geq 2m}\P(T^t\eta^{\beta_n}_1\neq T^t\eta^{\beta_n}_2)=0.
\end{equation}
\end{proposition}
We first prove \eqref{eq:escapedecorrelate} under the assumption that Proposition \ref{prop:couplingfiniteonesided} is true. Let $\varepsilon>0$. Then there exists $t_0$ such that for all $t\geq t_0$ and $\beta_n\geq 2t$,
\begin{equation}
	\begin{split}
		&\limm\P(\eta^{\alpha_m-\beta_n}[1,\alpha_m-\beta_n]\cap\eta^{\beta_n}[t,\beta_n]\neq\varnothing\mid\eta^{\alpha_m-\beta_n}(0)=\eta^{\beta_n}(\beta_n),\,\eta^{\beta_n}[0,k]=\zeta)\\
		=&\limm\P(\eta^{\alpha_m-\beta_n}[1,\alpha_m-\beta_n]\cap\eta^{\beta_n}[t,\beta_n]\neq\varnothing\mid\eta^{\alpha_m-\beta_n}(0)=\eta^{\beta_n}(\beta_n))\\
		=&\limm\P(\eta^{\alpha_m-\beta_n}[1,\alpha_m-\beta_n]\cap\eta^{\beta_n}[0,\beta_n]\neq\varnothing\mid\eta^{\alpha_m-\beta_n}(0)=\eta^{\beta_n}(\beta_n))\\
		&-\limm\P(\eta^{\alpha_m-\beta_n}[1,\alpha_m-\beta_n]\cap\eta^{\beta_n}[0,t]\neq\varnothing\mid\eta^{\alpha_m-\beta_n}(0)=\eta^{\beta_n}(\beta_n)\pm\varepsilon\\
		=&\limm\P(\eta^{\alpha_m-\beta_n}[1,\alpha_m-\beta_n]\cap\eta^{\beta_n}[0,\beta_n]\neq\varnothing\mid\eta^{\alpha_m-\beta_n}(0)=\eta^{\beta_n}(\beta_n))\pm2\varepsilon,
	\end{split}
\end{equation}
which was what we wanted to show.

\subsubsection{Proof of Proposition \ref{prop:couplingfiniteonesided}}
Consider a one-sided version of the finite two-sided coupling from Section \ref{sec:probabilityconv} or, equivalently, a finite version of the infinite coupling from Section \ref{sec:coupling}. So for some sequence $a_\l$, assuming that $\en_1$ and $\en_2$ are coupled, we sample an independent $n-a_\l$-step SAW $\eta^{n-a_\l}$. As before, if $\eta^{n-a_\l}$ escapes both $\en_1[0,a_\l]$ and $\en_2[0,a_\l]$, we let $(T^{a_\l}\en_1)[a_{\l+1},a_{\l}]=\eta^{n-a_\l}[0,a_{\l+1}-a_\l]=(T^{a_\l}\en_2)[a_{\l+1},a_{\l}]$. The cases where $\eta^{n-a_\l}$ escapes only one or none of the paths are similar to before.

The proof of Proposition \ref{prop:couplingfiniteonesided} using this coupling is very similar as in Section \ref{sec:ergodicityproof}. In Section \ref{sec:ergodicityproof}, the only place where the existence of infinite SAW is used is in the proof of Lemma \ref{lemma:avoidingconditioned}. It turns out that if we let $m$ be a subsequence of $(\beta_n)_n$, then a similar result holds using the convergence of the law of $(\eta^{\beta_n})_n$ to $P_2$. As a consequence, we must choose the $a_\l$ to be a subsequence of $(\beta_n)_n$. The rest of the proof then goes through as before.

\begin{lemma}
For all $\xi\in\SAW_k$,
\begin{equation}
	\begin{split}
		\liminf_{\substack{m\rightarrow\infty\\ m\in(\beta_n)_n}}\inf_{\beta_n\geq2m}\P(\eta^{\beta_n-m}\esc\eta^{\beta_n}[0,m]\mid\eta^{\beta_n}[0,k]=\xi)\geq A^{-1}
	\end{split}
\end{equation}
\end{lemma}
\begin{proof}
The proof is similar to the proof of Lemma \ref{lemma:avoidingconditioned}. Recall the notation $\zeta\succeq\xi$ if $\zeta$ is a self-avoiding path of length $m\geq k$ that starts with $\xi$. Then, by \eqref{eq:dmpaltsing},
\begin{equation}
	\begin{split}
		&\liminf_{\substack{m\rightarrow\infty\\ m\in(\beta_n)_n}}\inf_{\beta_n\geq2m}\P(\eta^{\beta_n-m}\esc\eta^n[0,m]\mid\eta^{\beta_n}[0,k]=\xi)\\
		=&\liminf_{\substack{m\rightarrow\infty\\ m\in(\beta_n)_n}}\inf_{\beta_n\geq2m}\sum_{\zeta\succeq\xi}\P(\eta^{\beta_n-m}\esc\zeta)\P(\eta^{\beta_n}[0,m]=\zeta)\P(\eta^{\beta_n}[0,k]=\xi)^{-1}\\
		=&\liminf_{\substack{m\rightarrow\infty\\ m\in(\beta_n)_n}}\inf_{\beta_n\geq2m}\P(\eta^{\beta_n}[0,k]=\xi)^{-1}\sum_{\zeta\succeq\xi}\P(\eta^{\beta_n}[0,m]=\zeta)^2\frac{c_{\beta_n}}{c_{\beta_n-m}}.
	\end{split}
\end{equation}
By applying Jensen's inequality to the last sum, we obtain that the above is bounded from below by
\begin{equation}
	\begin{split}
		&\liminf_{\substack{m\rightarrow\infty\\ m\in(\beta_n)_n}}\inf_{\beta_n\geq2m}\P(\eta^{\beta_n}[0,k]=\xi)^{-1}\frac{c_{\beta_n}}{c_{\beta_n-m}}c_m(\xi)\left(\sum_{\zeta\succeq\xi}\P(\eta^{\beta_n}[0,m]=\xi)c_m(\xi)^{-1}\right)^2\\
		=&\liminf_{\substack{m\rightarrow\infty\\ m\in(\beta_n)_n}}\inf_{\beta_n\geq2m}\frac{\P(\eta^{\beta_n}[0,k]=\xi)c_{\beta_n}}{c_m(\xi)c_{\beta_n-m}}\\
		=&\liminf_{\substack{m\rightarrow\infty\\ m\in(\beta_n)_n}}\inf_{\beta_n\geq2m}\frac{\P(\eta^{\beta_n}[0,k]=\xi)c_{\beta_n}}{\P(\eta^m[0,k]=\xi)c_mc_{\beta_n-m}}\\
		=&\liminf_{\substack{m\rightarrow\infty\\ m\in(\beta_n)_n}}\inf_{\beta_n\geq2m}\frac{\P(\eta^{\beta_n}[0,k]=\xi)A\mu^{\beta_n}}{\P(\eta^m[0,k]=\xi)A^2\mu^m\mu^{\beta_n-m}}=A^{-1}.
	\end{split}
\end{equation}
In the last line, we use that the laws of $(\eta^{\beta_n})_{n\in\N}$ and $(\eta^m)_{m\in(\beta_n)_n}$ both converge to $P_2$ and we use that $\beta_n\geq2m$, so $\beta_n-m\rightarrow\infty$.
\end{proof}
Thus, if we choose the sequence $(a_\l)_{\l\in\N}$ to grow sufficiently fast and to be a subsequence of~$(\beta_n)_{n\in\N}$, then the probability of successfully coupling at each iteration is bounded from below. The proof of this fact follows in the exact same way as the proof of Lemma \ref{lemma:probcouple}. Furthermore, the proof that the probability of staying successfully coupled is exponentially large for $(a_\l)_{\l\in\N}$ relies only on the hitting estimates of Section \ref{sec:hitting} and does not use existence of infinite SAW. So the proof is entirely analogous to Section \ref{sec:stayingcoupled}. The conclusion of the proof of Proposition \ref{prop:couplingfiniteonesided} then follows as in Section~\ref{sec:proofoferg}.

\subsubsection{Proof of Conjecture \ref{con:infalt}}
In the preceding, we proved existence of infinite SAW. It remains to show that the law of infinite SAW is the unique symmetric law satisfying \eqref{eq:dmpaltsing}. The proof is very similar to the existence proof and we omit many details.

Let $P$ be a symmetric probability measure on $\SAW_\infty$ satisfying \eqref{eq:dmpaltsing}. Then like before,
\begin{equation}
	\begin{split}
		P(\eta[0,k]=\zeta)=&\frac{c_n(\zeta)}{\mu^n}P(\eta\esc\eta^n\mid\eta^n[0,k]=\zeta)\\
		\sim&\P(\esing[0,k]=\zeta)AP(\eta\esc\eta^n\mid\eta^n[0,k]=\zeta),\qquad n\rightarrow\infty.
	\end{split}
\end{equation}
So we would like to show that the last probability tends to $A^{-1}$. Again, for all $m\leq n$,
\begin{equation}
    \begin{split}
        &P(\eta\esc\eta^n\mid\eta^n[0,k]=\zeta)\\
        =&1-P(\eta[1,\infty)\cap\eta^n[m,n]\neq\varnothing\mid\eta(0)=\eta^n(n),\,\eta^n[0,k]=\zeta)\\
        &-P(\eta[1,\infty)\cap\eta^n[0,m)\neq\varnothing,\,\eta[1,\infty)\cap\eta^n[m,n]=\varnothing\mid\esing(0)=\eta^n(n),\,\eta^n[0,k]=\zeta)
    \end{split}
\end{equation}
Note that $P(\eta\in\cdot)\leq A\P(\esing\in\cdot)$, which implies that the hitting estimates of Section \ref{sec:hitting} also hold for $P$. Hence, the last term tends to 0 as $n\rightarrow\infty$ for fixed $m$. Furthermore, by the coupling for $\eta^n$, we have $P(\eta[1,\infty)\cap\eta^n[m,n]\neq\varnothing\mid\eta(0)=\eta^n(n),\,\eta^n[0,k]=\zeta)\sim P(\eta[1,\infty)\cap\eta^n[0,n]\neq\varnothing\mid\eta(0)=\eta^n(n))$. So
\begin{equation}
	\begin{split}
		P(\eta\esc\eta^n\mid\eta^n[0,k]=\zeta)\sim&P(\eta\esc\eta^n)\\
		=&c_n^{-1}\sum_{\zeta\in\SAW_n}\P(\eta\esc\zeta)=c_n^{-1}\mu^n\sim A^{-1},\qquad n\rightarrow\infty.
	\end{split}
\end{equation}
which is what we wanted to show.

\paragraph{Acknowledgements} The author thanks Perla Sousi for her guidance, the many mathematical discussions and her helpful comments on this paper. The author also thanks Gordon Slade and Yucheng Liu for an insightful discussion. Finally, the author thanks Antal Járai and Wendelin Werner for their helpful comments on an earlier version of this paper. This work was supported by the University of Cambridge Harding Distinguished Postgraduate Scholarship Programme.

\paragraph{AI disclamer} No AI was used in proving or writing any part of the paper. AI was only used to help generate \Cref{fig:coupling}.

\bibliographystyle{hmralphaabbrv}
\bibliography{referenties}
\end{document}

%% file: TikZsuccessbefore.tex
\begin{tikzpicture}
    \coordinate (start) at (0,0);

    \draw[thick, blue]
        (start) .. controls (-0.3,0.3) and (0,0.5) .. (-0.5,0.5)
                .. controls (-0.75,0.5) and (-0.5,0.75) .. (-0.9,1)
                .. controls (-1.25,1.25) and (-1.5,1) .. (-1.6,1.4)
        node[pos=-0.3, left] {$\eta_{1}^{\infty}[0,k]$}; 

    \draw[thick, blue]
        (-1.7,1.355) -- (-1.5,1.415)
        node[pos=0.5, right] {$k$};

    \draw[thick, red]
        (start) .. controls (0.3,0.2) and (0.5,0.4) .. (0.75,0.5)
                .. controls (1,0.75) and (0.9,1) .. (1.25,1.25)
                .. controls (1.75,1.5) and (1.9,1.75) .. (1.85,2)
        node[pos=0.15, right] {$\eta_{2}^{\infty}[0,k]$};

    \draw[thick, red]
        (1.75,1.955) -- (1.95,2.015)
        node[pos=0.5, right] {$k$};

    \draw[dashed, olive]
        (-1.6,1.4) .. controls (-1.9,2.1) and (-1.5,2.5) .. (-2.5,2.8)
                    .. controls (-3.5,3.2) and (-3,4) .. (-3.5,4.5); 

    \draw[thick, olive]
        (-2.55,2.68)  -- (-2.45,2.92)
        node[pos=0.3, above right] {$a_1$}; 

    \draw[dashed, olive]
        (1.85,2) .. controls (1.55,2.7) and (1.95,3.1) .. (0.95,3.4)
                  .. controls (-0.05,3.8) and (0.45,4.6) .. (-0.05,5.1);

    \draw[thick, olive]
        (0.9,3.28)  -- (1.0,3.52)
        node[pos=0.3, above right] {$a_1$}; 

    \draw[dashed, olive, ->] 
        (-3.5,4.5) -- ++(-0.5,0.5)
        node[above] {$\esing_{0,1}$};

    \draw[dashed, olive, ->] 
        (-0.05,5.1) -- ++(-0.5,0.5)
        node[above] {$\esing_{0,1}$};

    \fill (start) circle (2pt);

\end{tikzpicture}

%% file: TikZsuccessafter.tex
\begin{tikzpicture}
    \coordinate (start) at (0,0);

    \draw[thick, blue]
        (start) .. controls (-0.3,0.3) and (0,0.5) .. (-0.5,0.5)
                .. controls (-0.75,0.5) and (-0.5,0.75) .. (-0.9,1)
                .. controls (-1.25,1.25) and (-1.5,1) .. (-1.6,1.4)
        node[pos=-0.3, left] {$\eta_{1}^{\infty}[0,a_1]$}; 

    \draw[thick, blue]
    (-1.7,1.355) -- (-1.5,1.415)
    node[pos=0.5, right] {$k$};

    \draw[thick, red]
        (start) .. controls (0.3,0.2) and (0.5,0.4) .. (0.75,0.5)
                .. controls (1,0.75) and (0.9,1) .. (1.25,1.25)
                .. controls (1.75,1.5) and (1.9,1.75) .. (1.85,2)
        node[pos=0.15, right] {$\eta_{2}^{\infty}[0,a_1]$};

    \draw[thick, red]
        (1.75,1.955) -- (1.95,2.015)
        node[pos=0.5, right] {$k$};

    \draw[thick, blue]
        (-1.6,1.4) .. controls (-1.9,2.1) and (-1.5,2.5) .. (-2.5,2.8); 

    \draw[thick, white]
        (-2.5,2.8) .. controls (-3.5,3.2) and (-3,4) .. (-3.5,4.5); 

    \draw[thick, blue]
        (-2.55,2.68)  -- (-2.45,2.92)
        node[pos=0.3, above right] {$a_1$}; 

    \draw[thick, red]
        (1.85,2) .. controls (1.55,2.7) and (1.95,3.1) .. (0.95,3.4); 

    \draw[thick, white]
        (0.95,3.4) .. controls (-0.05,3.8) and (0.45,4.6) .. (-0.05,5.1); 

    \draw[thick, red]
        (0.9,3.28)  -- (1.0,3.52)
        node[pos=0.3, above right] {$a_1$}; 

    \draw[thick, white, ->] 
        (-3.5,4.5) -- ++(-0.5,0.5)
        node[above] {$\esing_{0,1}$};

    \draw[thick, white, ->] 
        (-0.05,5.1) -- ++(-0.5,0.5)
        node[above] {$\esing_{0,1}$};

    \fill (start) circle (2pt);

\end{tikzpicture}

%% file: TikZfailbefore.tex
\begin{tikzpicture}
    \coordinate (start) at (0,0);

    \draw[thick, blue]
        (start) .. controls (-0.3,0.3) and (0,0.5) .. (-0.5,0.5)
                .. controls (-0.75,0.5) and (-0.5,0.75) .. (-0.9,1)
                .. controls (-1.25,1.25) and (-1.5,1) .. (-1.6,1.4)
        node[pos=-0.3, left] {$\eta_{1}^{\infty}[0,k]$};

    \draw[thick, blue]
    (-1.7,1.355) -- (-1.5,1.415)
    node[pos=0.5, right] {$k$};

    \draw[thick, red]
        (start) .. controls (0.3,0.2) and (0.5,0.4) .. (0.75,0.5)
                .. controls (1,0.75) and (0.9,1) .. (1.25,1.25)
                .. controls (1.75,1.5) and (1.9,1.75) .. (1.85,2)
        node[pos=0.15, right] {$\eta_{2}^{\infty}[0,k]$};

    \draw[thick, red]
        (1.75,1.955) -- (1.95,2.015)
        node[pos=0.5, right] {$k$};

    \draw[dashed, olive]
    (-1.6,1.4) .. controls (-1.6,1.6) and (-1.8,1.8) ..  (-2.0,1.7)
                .. controls (-2.55,1.4) and (-2.75,1.1) .. (-2.9,0.8) 
                .. controls (-3.15,0.2) and (-2.4,-0.6) .. (-1.95,-1.0)
                .. controls (-1.65,-1.2) and (-1.45,-1.6) .. (-1.45,-2.1);

    \draw[dashed, olive]
    (1.85,2.0) .. controls (1.85,2.2) and (1.65,2.4) ..  (1.45,2.3)
                .. controls (0.9,2.0) and (0.7, 1.7) .. (0.55,1.4) 
                .. controls (0.3,0.8) and (1.05,0) .. (1.5,-0.4)
                .. controls (1.8,-0.6) and (2.0,-1.0) .. (2.0,-1.5);

    \draw[thick, olive]
    (-2.05,-1.1) -- (-1.85,-0.9)
    node[pos=0.3, above right] {$a_1$};

    \draw[thick, olive]
    (1.4,-0.5) -- (1.6,-0.3)
    node[pos=0.3, above right] {$a_1$};

    \draw[dashed, olive, ->] 
        (-1.45,-2.1) -- ++(0,-0.5)
        node[left] {$\esing_{0,1}$};

    \draw[dashed, olive, ->] 
        (2.0,-1.5) -- ++(0,-0.5)
        node[left] {$\esing_{0,1}$};

    \draw[thick, white]
        (1.85,2) .. controls (1.9,2.3) and (2.2,2.6) .. (2.0,3.0)
              .. controls (1.8,3.4) and (2.3,3.6) .. (2.5,3.8)
              .. controls (2.7,4.0) and (3.2,4.3) .. (3.4,4.3);

    \fill (start) circle (2pt);

\end{tikzpicture}

%% file: TikZfailafter.tex
\begin{tikzpicture}
    \coordinate (start) at (0,0);

    \draw[thick, blue]
        (start) .. controls (-0.3,0.3) and (0,0.5) .. (-0.5,0.5)
                .. controls (-0.75,0.5) and (-0.5,0.75) .. (-0.9,1)
                .. controls (-1.25,1.25) and (-1.5,1) .. (-1.6,1.4)
        node[pos=-0.3, left] {$\eta_{1}^{\infty}[0,a_1]$};

    \draw[thick, blue]
    (-1.7,1.355) -- (-1.5,1.415)
    node[pos=0.5, right] {$k$};

    \draw[thick, red]
        (0,0) .. controls (0.3,0.2) and (0.5,0.4) .. (0.75,0.5)
              .. controls (1,0.75) and (0.9,1) .. (1.25,1.25)
              .. controls (1.75,1.5) and (1.9,1.75) .. (1.85,2)
        node[pos=0.15, right] {$\eta_{2}^{\infty}[0,a_1]$};

    \draw[thick, red]
        (1.85,2) .. controls (1.9,2.3) and (2.2,2.6) .. (2.0,3.0)
              .. controls (1.8,3.4) and (2.3,3.6) .. (2.5,3.8)
              .. controls (2.7,4.0) and (3.2,4.3) .. (3.4,4.3);

    \draw[thick, red]
        (3.4,4.4) -- (3.4,4.2)
            node[pos=0.65, above] {$a_1$};

    \draw[thick, red]
        (1.75,1.955) -- (1.95,2.015)
        node[pos=0.5, right] {$k$};

    \draw[thick, blue]
    (-1.6,1.4) .. controls (-1.6,1.6) and (-1.8,1.8) ..  (-2.0,1.7)
                .. controls (-2.55,1.4) and (-2.75,1.1) .. (-2.9,0.8) 
                .. controls (-3.15,0.2) and (-2.4,-0.6) .. (-1.95,-1.0);

    \draw[thick, blue]
    (-2.05,-1.1) -- (-1.85,-0.9)
    node[pos=0.3, above right] {$a_1$};

    \draw[thick, white]
    (-1.95,-1.0) .. controls (-1.65,-1.2) and (-1.45,-1.6) .. (-1.45,-2.1);

    \draw[thick, white]
    (1.85,2.0) .. controls (1.85,2.2) and (1.65,2.4) ..  (1.45,2.3)
                .. controls (0.9,2.0) and (0.7, 1.7) .. (0.55,1.4) 
                .. controls (0.3,0.8) and (1.05,0) .. (1.5,-0.4)
                .. controls (1.8,-0.6) and (2.0,-1.0) .. (2.0,-1.5);

    \draw[dashed, white, ->] 
        (-1.45,-2.1) -- ++(0,-0.5)
        node[left] {$\esing_{0,1}$};

    \draw[dashed, white, ->] 
        (2.0,-1.5) -- ++(0,-0.5)
        node[left] {$\esing_{0,1}$};

    \fill (start) circle (2pt);

\end{tikzpicture}